\documentclass[a4paper]{amsart}

\usepackage{amsmath,amsthm,amssymb,enumerate}
\usepackage{epsfig}
\usepackage{amssymb}
\usepackage{amsmath}
\usepackage{amssymb}
\usepackage{amsmath,amsthm}
\usepackage[latin1]{inputenc}
\usepackage[T1]{fontenc}
\usepackage{path}
\usepackage{ae,aecompl}
\usepackage{amsfonts}
\usepackage{amsxtra}
\usepackage{euscript,mathrsfs}
\usepackage{color}
\usepackage[left=3.4cm,right=3.4cm,top=4cm,bottom=4cm]{geometry}
\usepackage[citecolor=blue,colorlinks=true]{hyperref}
\allowdisplaybreaks

\usepackage{enumitem}
\setenumerate{label={\rm (\alph{*})}}

\usepackage{amsfonts}
\usepackage{amsxtra}

\usepackage[frak,hyperref,theorem_section]{paper_diening}
\numberwithin{equation}{section}

\newcommand{\Div}{\divergence}

\newcommand{\R}{\mathbb R}
\newcommand{\N}{\mathbb N}

\newcommand{\E}{\mathbb E}
\newcommand{\p}{\mathbb P}

\newcommand{\F}{\mathfrak F}

\newcommand{\A}{\mathcal A}

\newcommand{\dd}{\mathrm d}
\newcommand{\dx}{\, \mathrm{d}x}

\newcommand{\ds}{\, \mathrm{d}\sigma}
\newcommand{\dt}{\, \mathrm{d}t}
\newcommand{\dxt}{\,\mathrm{d}x\, \mathrm{d}t}

\newcommand{\dif}{\mathrm{d}}

\newcommand{\mf}{\mathfrak{F}}

\newcommand{\prst}{\mathbb{P}}

\newcommand{\mn}{\mathbb{N}}

\newcommand{\mt}{\mathcal{O}}
\newcommand{\tor}{\mathcal{O}}

\DeclareMathOperator{\diver}{div}

\begin{document}

\title[2D Navier--Stokes with additive noise]
{Numerical analysis of 2D Navier--Stokes equations with additive stochastic forcing}

\author{Dominic Breit}
\address{Department of Mathematics, Heriot-Watt University, Riccarton Edinburgh EH14 4AS, UK}
\email{d.breit@hw.ac.uk}

\author{Andreas Prohl}
\address{Mathematical Institute, University of T\"ubingen, Auf der Morgenstelle 10,
72076 T\"ubingen,
Germany }
\email{prohl@na.uni-tuebingen.de}

%
%
%\author{Martina Hofmanov\'a}
%\address[M. Hofmanov\'a]{Fakult\"at f\"ur Mathematik, Universit\"at Bielefeld, D-33501 Bielefeld, Germany}
%\email{hofmanova@math.uni-bielefeld.de}

\begin{abstract}
We propose and study a temporal, and spatio-temporal discretisation  of the 2D stochastic Navier--Stokes equations in bounded domains supplemented with no-slip boundary conditions. Considering additive noise, we base its construction on the related nonlinear random PDE,
which is solved by a transform of the solution of the stochastic Navier--Stokes equations. We show {\em strong} rate (up to) $1$ in probability for a corresponding discretisation in space and time (and space-time). Convergence of order (up to) 1 in time was previously only known for linear SPDEs.
%
%
%The main result includes a weak convergence rate (up to) $1$ for the {\color{red} spatio-}temporal discretization with respect to convergence in probability.
%Our approach is based on a truncated problem and will be suitable also for other semi-linear SPDEs with non-Lipschitz nonlinearity.

\end{abstract}

\keywords{Stochastic Navier--Stokes equations \and error in probability \and space-time discretisation  \and convergence rates}
\subjclass[2010]{65M15, 65C30, 60H15, 60H35}

\date{\today}

\maketitle

%
%
%\author{Dominic Breit}
%\ead{Breit@math.lmu.de}
%
%
%\address{LMU Munich,
%Mathematical Institute,
%Theresienstra\ss e 39,
%80333 Munich -- Germany}
%
%
%
%  \begin{abstract}
%We consider systems of stochastic evolutionary equations of the type
%$$d\bfu=\Div\,\bfS(\nabla \bfu)\,dt+\Phi(\bfu)d\bfW_t$$
%where $\bfS$ is a non-linear operator, for instance the $p$-Laplacian
%$$\bfS(\bfxi)=(\kappa+|\bfxi|)^{p-2}\bfxi,\quad \bfxi\in\R^{d\times D},$$
%with $p\in(1,\infty)$ and $\Phi$ grows linearly. We extend known results about 
%the deterministic problem to the stochastic situation. 
%  \end{abstract}
%
%  \begin{keyword}
%   Parabolic stochastic PDE's; Non-linear Laplacian-type systems; Existence of weak solutions; Regularity of solutions;\\
%
%   2010 MSC: 35R60\sep 35D30\sep
%    60H15\sep 35K55 \sep 35B65 
%  \end{keyword}
%
%\end{frontmatter}

\section{Introduction}

We are concerned with the numerical approximation of the 2D stochastic Navier--Stokes equations in a bounded smooth domain $\mt\subset\R^2$ with no-slip boundary conditions.
They describe the flow of a homogeneous incompressible fluid in terms of the velocity field $\bfu$ and pressure function $p$ --- which are both defined on a filtered probability space $(\Omega,\mathfrak F,(\mathfrak F_t),\p)$, and read as
\begin{align}\label{eq:SNS}
\left\{\begin{array}{rc}
\dd\bfu=\mu\Delta\bfu\dt-(\nabla\bfu)\bfu\dt-\nabla p\dt+\Phi\,\dd W
& \mbox{in $\mathcal Q_T$,}\\
\Div \bfu=0\qquad\qquad\qquad\qquad\qquad\,\,\,\,& \mbox{in $\mathcal Q_T$,}\\
%\bfu=0\qquad\qquad\qquad\qquad\qquad\,\,\,\,& \mbox{on $(0,T)\times\mt$,}\\
\bfu(0)=\bfu_0\,\qquad\qquad\qquad\qquad\qquad&\mbox{ \,in $\mt$,}\end{array}\right.
\end{align}
$\p$-a.s. in $\mathcal Q_T:=(0,T)\times\mt$, with terminal time $T>0$, the viscosity $\mu>0$, and $\bfu_0$ a given initial datum. The momentum equation is driven by a cylindrical Wiener process $W$ and the diffusion coefficient $\Phi$ takes values in the space of Hilbert-Schmidt operators; see Section \ref{sec:prob} for details.

The existence, regularity, and long-time behaviour of (probabilistically) strong solutions to \eqref{eq:SNS} have been studied extensively over the last three decades, and we refer to \cite{KukShi} for a complete picture. In most of the available numerical works for \eqref{eq:SNS} which address {\em strong} convergence analysis (we refer to {\em e.g.}, \cite{CP,BrDo,BeMi1,BeMi2}) periodic boundary conditions are assumed. This enables $\int (\nabla {\bf u}) {\bf u} \cdot \Delta {\bf u} \, {\rm d}x= 0$ as a key property to validate higher moment bounds for the solution 
 of (\ref{eq:SNS}) in strong spatial norms, for this and related regularity properties we refer to~\cite{CC,Ca,KukShi}. The convergence analysis for related discretisation schemes then has to cope with the interplay of the quadratic nonlinearity with the (possibly multiplicative) noise term in \eqref{eq:SNS}: without the noise term, a standard Gronwall-argument suffices to validate optimal convergence rates for a (finite element) based space-time discretisation of (\ref{eq:SNS}), see {\em e.g.} \cite{HeyRa1}. This argument
may not be applied directly any more in the presence of noise, where only {\em moments} of solutions of (\ref{eq:SNS}) are available (and evenly so for a stable space-time discretisation), rather than {\em uniform} bounds with respect to $\omega \in \Omega$:
%% --- to bound error effects from the nonlinearity: 
 in \cite{CP}, the convergence analysis therefore separately estimates error amplification through  the nonlinearity in (\ref{tdiscr-add}) on ``large'' nested local sets $( \Omega_{t_{m-1}})_m \subset \Omega$, and proves
``smallness'' of errors with the help of discrete stability properties on the complementary ``small'' sets $(\Omega \setminus \Omega_{t_{m-1}})_m$. This approach has been further refined in the works \cite{BrDo,BeMi1,BeMi1a,BeMi2}. 

The works \cite{CP,BrDo,BeMi1,BeMi1a,BeMi2} mainly address \eqref{eq:SNS} with (Lipschitz) multiplicative noise, and use the concept of {\em (strong) orders in probability} for the error analysis of related space-time discretisations; this concept
is due to the nonlinear character of the problem, the effect of which is controlled
along (families of) local subsets of $\Omega$ as mentioned above. We remark that it  is weaker compared to {\em mean-square convergence}, which is usually employed to bound discretisation errors for linear stochastic PDEs. 
A parallel numerical program for (\ref{eq:SNS}) when the noise is  {\em additive} started in \cite{BeMi1,BeMi1a,BeMi2}: these works address the question if improved convergence results may be obtained for the same schemes, by exploiting additionally available
improved stability properties for this data setting (``exponential moment bounds''). It turns out that this is the case, provided that the strength of noise is small
with respect to the viscosity $\mu$: the first result is \cite[Thm.~3.4]{BeMi1a}, which 
proves a {\em mean-square convergence} rate, which crucially depends on the ratio of $\mu$, and the strength of noise: for example, the latter is limited to be of order ${O}(\mu)$ for $\mu \ll 1$; another result in this direction is
asymptotic mean-square convergence order (up to $\frac{1}{2}$) in \cite[Thm. 3.3]{BeMi2}, provided that the strength of noise is  of order ${O}(\mu)$.

In this work, we provide a new temporal (see (\ref{tdiscr-add})) and spatio-temporal (see (\ref{tdiscr-add-fem1})--(\ref{tdiscr-add-fem2})) discretisation for (\ref{eq:SNS}) with additive noise, which is based on the reformulation (\ref{eq:SNSy}) as {\em random PDE} for the transform ${\bf y} = {\bf u} - \Phi W$. The motivation is to therefore settle improved convergence behaviour of approximates for ${\bf y}$, which has improved time regularity properties: we will see in Lemma \ref{lem:regadditive} that  ${\bf y}$ is differentiable in time, while ${\bf u}$ is only H\"older continuous, with exponent $\alpha < \frac{1}{2}$. In fact, we show strong rate (up to) $1$ in probability for (\ref{tdiscr-add}) --- which doubles the order 
in \cite{BrDo} for the space-time discretisation in \cite{CP} for (\ref{eq:SNS})   with multiplicative noise; see Theorem \ref{thm:maintdiscr} for semi-discretisation (\ref{tdiscr-add}), and the generalisation of this result in Section \ref{sec:fe} for the spatio-temporal discretisation (\ref{tdiscr-add-fem1})--(\ref{tdiscr-add-fem2}).
To our knowledge, an (up to) first-order temporal discretisation
for a nonlinear SPDE was not available in the literature before, and we expect that the conceptual steps
for its construction and convergence analysis  given in the sections below are applicable to other
nonlinear SPDEs with additive noise as well.
The error analysis below is performed in the presence of homogeneous Dirichlet boundary data in (\ref{eq:SNS}), and driving trace-class, solenoidal noise which vanishes on the boundary. We emphasise that our approach does not require any smallness condition on the noise. It is also worth to point out that our convergence results improve to a {\em mean-square convergence} rate $1$ for the stochastic Stokes equations with additive noise, see Remark \ref{rem:A}.

We recall that $\int (\nabla {\bf u}) {\bf u} \cdot \Delta {\bf u} \, {\rm d}x= 0$ does not hold any more for this boundary value problem for (\ref{eq:SNS}), and also higher moment bounds
for ${\bf u}$ in strong norms seem beyond reach, which then affects the efficiency 
of the strategy of proof in \cite{CP,BrDo,BeMi1,BeMi1a,BeMi2} which 
separately studies errors on ``large'' nested local sets $( \Omega_{t_{m-1}})_m \subset \Omega$, and small ones. Therefore, a different strategy to verify rates with the help of (discrete) stopping times and 
truncated nonlinearities was developed in \cite{BrPr} for (a spatial discretisation of) (\ref{tdiscr}) and iterates $({\bf u}_m)_m$, which was motivated by the analytical works \cite{Mi,GlZi}. We use related concepts in Section \ref{sec:add}  to
establish stability and convergence rates for (solenoidal) iterates $({\bf y}_m)_m$ with zero boundary data from the temporal semi-discretisation (\ref{tdiscr-add}), which satisfy  ${\bf y}_m = 
{\bf u}_m - \Phi W(t_m)$. It is worth to point out that we work here under more restrictive assumptions on the noise compared to our previous paper \cite{BrPr}; this allows us to prove higher order estimates for the temporal
approximations in Lemma \ref{lemma:3.1}, which turns out to be useful for the 
error analysis of $({\bf y}_m)_m$ from (\ref{tdiscr-add}). For this it is crucial to assume that the diffusion coefficient vanishes on $\partial\mathcal O$. Similarly, the analysis in Section \ref{sec:3.1} heavily relies on 
the solenoidal character of the noise (conceptionally, this assumption can be dropped in the case of periodic boundary conditions at least as far as the purely temporal discretisation is concerned).

Section 
\ref{sec:fe} then proposes the LBB-stable mixed finite element scheme
(\ref{tdiscr-add-fem1})--(\ref{tdiscr-add-fem2}) for iterates $({\bf Y}_m)_m$, for which we quantify the error $\max_{m}\Vert {\bf y}_m - {\bf Y}_m\Vert_{L^2_x}$, in particular: its proof rests on essentially considering a ``finite element version'' $({\bf U}_m)_m$ in (\ref{fem-9}) of $({\bf u}_m)_m$, which solves (\ref{txdiscrpia}) as spatial discretisation of (\ref{tdiscr}). To proceed like this here, {\em i.e.,} to focus on the sequence $({\bf U}_m)_m$ rather than $({\bf Y}_m)_m$, is justified for the {\em spatial} error analysis part, where the limited temporal regularity of the driving Wiener process does not spoil the order. The main result is then given in Theorem \ref{thm:maintdiscra}, which rests on discrete stability results for (\ref{tdiscr}) that are collected in Section \ref{sec:disest}.

\section{Mathematical framework}
\label{sec:framework}

\subsection{Probability setup}\label{sec:prob}

Let $(\Omega,\F,(\F_t)_{t\geq0},\prst)$ be a stochastic basis with a complete, right-continuous filtration. The process $W$ is a cylindrical Wiener process, that is, $W(t)=\sum_{k\geq1}\beta_j(t) e_j$ with $(\beta_j)_{j\geq1}$ being mutually independent real-valued standard Wiener processes relative to $(\F_t)_{t\geq0}$, and $(e_j)_{j\geq1}$ a complete orthonormal system in a separable Hilbert space $\mathfrak{U}$.
We assume that the diffusion coefficient $\varPhi$ belongs to the set of Hilbert-Schmidt operators $L_2(\mathfrak U;\mathbb H)$, where
$\mathbb H$ can take the role of various Hilbert spaces. Of particular importance are spaces of solenoidal vector fields, such as $L^2_{\Div}(\mt;\R^2)$ and $W^{1,2}_{0,\Div}(\mt;\R^2)$. They are defined as the closure of the solenoidal $C_{c}^\infty(\mt;\R^2)$-functions\footnote{We denote this space by $C_{c,\Div}^\infty(\mt;\R^2)$ in the following.} with respect to the $L^2(\mt;\R^2)$- or $W^{1,2}_0(\mt;\R^2)$-norm, respectively. 
Given $\Phi\in L^2(\mathfrak U;L^2_{\Div}(\mt;\R^2))$ the stochastic integral $$t\mapsto\int_0^t\varPhi\,\dif W$$
is a well defined process taking values in $L^2_{\Div}(\mt;\R^2)$ (see \cite{PrZa} for a detailed construction). Moreover, we can multiply by test functions to obtain
 \begin{align*}
\bigg\langle\int_0^t \varPhi\,\dd W,\bfphi\bigg\rangle_{L^2_x}=\sum_{j\geq 1} \int_0^t\langle \varPhi e_j,\bfphi\rangle_{L^2_x}\,\dd\beta_j \quad \forall\, \bfphi\in L^2(\mt;\R^2).
\end{align*}
Similarly, we can define stochastic integrals with values in $W^{1,2}_{0,\Div}(\mt;\R^2)$ and $ W^{2,2}(\mt;\R^2)$, respectively, if $\Phi$ belongs to the corresponding class of Hilbert-Schmidt operators. 
%The following lemma is a helpful tool to analysis the time-regularity
%of stochastic integrals (see, e.g., \cite[Lemma 9.1.3. b)]{Br2} or \cite[Lemma 4.6]{Ho1}).
%\begin{lemma}
%\label{lems:Holder}
% Let $\psi\in L^r(\Omega;L^r(0,T;L_2(\mathfrak U,L^2(\mt))))$, $r>2$, by an $(\mathfrak F_t)$-progressively measureable process and $W$ a cylindrical $(\F_t)$-Wiener process on $\mathfrak U$. Then the paths of the process $Z_t:=\int_0^t \psi\,\dd W$ are $\p$-a.s. H\"older continuous with exponent $\alpha\in\big(\frac{1}{r},\frac{1}{2}\big)$ and it holds
%\begin{align*}
%\E\Big[\|Z\|_{C^{\alpha}([0,T];L^2(\mt))}^r\Big]\leq c_\alpha\,\E\bigg[\sup_{0\leq t\leq T}\|\psi\|_{L_2(\mathfrak U,L^2(\mt))}^2\dt\bigg]^{\frac{r}{2}}.
%\end{align*}
%\end{lemma}

%The initial datum may be random in general, i.e. $\F_0$-measurable, and we assume at least $\bfu_0\in L^2(\Omega;L^2(\mt))$.

\subsection{The concept of solutions}
\label{subsec:solution}
In dimension two, pathwise uniqueness for analytically weak solutions of \eqref{eq:SNS} is known; we refer the reader for instance to Capi\'nski--Cutland \cite{CC}, Capi\'nski \cite{Ca}. Consequently, we may work with the definition of a weak pathwise solution.

\begin{definition}\label{def:inc2d}
Suppose that $\Phi\in L_2(\mathfrak U;L^2_{\Div}(\mt;\R^2))$.
Let $(\Omega,\mf,(\mf_t)_{t\geq0},\prst)$ be a given stochastic basis with a complete right-con\-ti\-nuous filtration and an $(\mf_t)$-cylindrical Wiener process $W$. Let $\bfu_0$ be an $\mf_0$-measurable random variable with values in $L^2_{\Div}(\mt;\R^2)$. Then $\bfu$ is called a \emph{weak pathwise solution} \index{incompressible Navier--Stokes system!weak pathwise solution} to \eqref{eq:SNS} with the initial condition $\bfu_0$ provided
\begin{enumerate}
\item the velocity field $\bfu$ is $(\mf_t)$-adapted and
$$\bfu \in C([0,T];L^2_{\diver}(\tor;\R^2))\cap L^2(0,T; W^{1,2}_{0,\diver}(\tor;\R^2))\quad\text{$\p$-a.s.},$$
% and
%\begin{align*}
%\E\bigg[\sup_{t\in[0,T]}\|\bfu\|^2_{L^2_x}\bigg]^p+\E\bigg[\bigg( \int_0^T \| \bfu \|^2_{W^{1,2}_x} \ \dt \bigg)^p \bigg] < \infty\ \mbox{for all}\ 1 \leq p < \infty;
%\end{align*}
%\item the velocity $\bfu$ is $(\mf_t)$-adapted and
%$$\bfu\in L^2(\Omega;L^2(0,T;W_{\text{div}}^{1,2}(\mt)))\cap L^2(\Omega;C_w([0,T];L^2_{\text{div}}(\mt))),$$
%\item $\bfu(0)=\bfu_0$ $\p$-a.s.,
\item the momentum equation
\begin{align*}
\int_{\mt}\bfu(t)&\cdot\bfvarphi\dx-\int_{\mt}\bfu_0\cdot\bfvarphi\dx
\\&=\int_0^t\int_{\mt}\bfu\otimes\bfu:\nabla\bfvarphi\dx\,\dif t-\mu\int_0^t\int_{\mt}\nabla\bfu:\nabla\bfvarphi\dx\,\dif s+\int_{\mt}\int_0^t\Phi\cdot\bfvarphi\,\dif W\dx
\end{align*}
holds $\p$-a.s. for all $\bfvarphi\in W^{1,2}_{0,\diver}(\mt;\R^2)$ and all $t\in[0,T]$.
\end{enumerate}
\end{definition}
\begin{theorem}\label{thm:inc2d}
Suppose that $\Phi\in L_2(\mathfrak U;L^2_{\Div}(\mt;\R^2))$. Let $(\Omega,\mf,(\mf_t)_{t\geq0},\prst)$ be a given stochastic basis with a complete right-continuous filtration and an $(\mf_t)$-cylindrical Wiener process $W$. Let $\bfu_0$ be an $\mf_0$-measurable random variable such that $\bfu_0\in L^r(\Omega;L^2_{\mathrm{div}}(\mt;\R^2))$ for some $r>2$. Then there exists a unique weak pathwise solution to \eqref{eq:SNS} in the sense of Definition \ref{def:inc2d} with the initial condition $\bfu_0$.
\end{theorem}

We give the definition of a strong pathwise solution to \eqref{eq:SNS} which exists up to a stopping time $\mathfrak t$. The velocity field belongs $\p$-a.s. to $C([0,\mathfrak t];W^{1,2}_{0,\diver}(\mt;\R^2))$.

\begin{definition}\label{def:strsol}
Let $(\Omega,\mf,(\mf_t)_{t\geq0},\prst)$ be a given stochastic basis with a complete right-continuous filtration and an $(\mf_t)$-cylindrical Wiener process $W$. Let $\bfu_0$ be an $\mf_0$-measurable random variable with values in $W^{1,2}_{0,\diver}(\mt;\R^2)$. The tuple $(\bfu,\mathfrak t)$ is called a \emph{local strong pathwise solution} \index{incompressible Navier--Stokes system!weak pathwise solution} to \eqref{eq:SNS} with the initial condition $\bfu_0$ provided
\begin{enumerate}
\item $\mathfrak t$ is a $\p$-a.s. strictly positive $(\mathfrak F_t)$-stopping time;
\item the velocity field $\bfu$ is $(\mf_t)$-adapted and
$$\bfu(\cdot\wedge \mathfrak t) \in C_{\mathrm loc}([0,\infty);W^{1,2}_{0,\diver}(\mt;\R^2))\cap L^2_{\mathrm loc}(0,\infty;W^{2,2}(\mt;\R^2)) \quad\text{$\p$-a.s.},$$
% and
%\begin{align*}
%\E\bigg[\sup_{t\in[0,T]}\|\bfu\|^2_{L^2_x}\bigg]^p+\E\bigg[\bigg( \int_0^T \| \bfu \|^2_{W^{1,2}_x} \ \dt \bigg)^p \bigg] < \infty\ \mbox{for all}\ 1 \leq p < \infty;
%\end{align*}
%\item the velocity $\bfu$ is $(\mf_t)$-adapted and
%$$\bfu\in L^2(\Omega;L^2(0,T;W_{\text{div}}^{1,2}(\mt)))\cap L^2(\Omega;C_w([0,T];L^2_{\text{div}}(\mt))),$$
%\item $\bfu(0)=\bfu_0$ $\p$-a.s.,
\item the momentum equation
\begin{align}\label{eq:mom}
\begin{aligned}
&\int_{\mt}\bfu(\tau\wedge \mathfrak t)\cdot\bfvarphi\dx-\int_{\mt}\bfu_0\cdot\bfvarphi\dx
\\&=-\int_0^{\tau\wedge\mathfrak t}\int_{\mt}(\nabla\bfu)\bfu\cdot\bfvarphi\dx\,\dif s+\mu\int_0^{\tau\wedge\mathfrak t}\int_{\mt}\Delta\bfu\cdot\bfvarphi\dx\,\dif s+\int_{\mt}\int_0^{\tau\wedge\mathfrak t}\Phi\cdot\bfvarphi\,\dif W\dx
\end{aligned}
\end{align}
holds $\p$-a.s. for all $\bfvarphi\in C^{\infty}_{c,\diver}(\mt;\R^2)$ and all $\tau\geq0$.
\end{enumerate}
\end{definition}
We finally define a maximal strong pathwise solution.
\begin{definition}[Maximal strong pathwise solution]\label{def:maxsol}
Fix a stochastic basis with a cylindrical Wiener process and an initial condition as in Definition \ref{def:strsol}. A triplet $$(\bfu,(\mathfrak{t}_R)_{R\in\N},\mathfrak{t})$$ is a maximal strong pathwise solution to system \eqref{eq:SNS} provided

\begin{enumerate}
\item $\mathfrak{t}$ is a $\p$-a.s. strictly positive $(\mathfrak{F}_t)$-stopping time;
\item $(\mathfrak{t}_R)_{R\in\mn}$ is an increasing sequence of $(\mathfrak{F}_t)$-stopping times such that
$\mathfrak{t}_R<\mathfrak{t}$ on the set $[\mathfrak{t}<\infty]$,
%$\mathfrak{t}_R<\mathfrak{t}$ on $[\mathfrak{t}<T]$ and 
$\lim_{R\to\infty}\mathfrak{t}_R=\mathfrak t$ $\p$-a.s., and
\begin{equation}\label{eq:blowup}
\mathfrak t_R:=\inf \big\{t\in[0,\infty):\,\,\|\bfu(t)\|_{W^{1,2}_x}\geq R\big\}\quad \text{on}\quad [\mathfrak{t}<\infty] ,
\end{equation}
with the convention that $\mathfrak{t}_R=\infty$ if the set above is empty;
\item each triplet $(\bfu,\mathfrak{t}_R)$, $R\in\mn$,  is a local strong pathwise solution in the sense  of Definition \ref{def:strsol}.
\end{enumerate}
\end{definition}
We talk about a global solution if we have (in the framework of Definition \ref{def:maxsol}) $\mathfrak t=\infty$ $\p$-a.s.
The following result is shown in \cite{Mi} (see also \cite{GlZi} for a similar statement).
\begin{theorem}\label{thm:inc2dmax}
Suppose that $\Phi\in L_2(\mathfrak U;W^{1,2}_{0,\Div}(\mt;\R^2))$ and
 $\bfu_0\in L^2(\Omega,W^{1,2}_{0,\diver}(\mt;\R^2))$. Then there is a
unique global maximal strong pathwise solution to \eqref{eq:SNS} in the sense  of Definition \ref{def:maxsol}.
\end{theorem}

\subsection{Regularity of solutions}
The following lemma is a special case of \cite[Lemma 3.1]{BrPr}, where also multiplicative noise is considered.
\begin{lemma}\label{lem:reg}
 \begin{enumerate}
\item[(a)] Assume that $\bfu_0\in L^r(\Omega,L^{2}_{\Div}(\mt;\R^2))$ for some $r\geq2$ and that $\Phi\in L_2(\mathfrak U;L^2_{\Div}(\mt;\R^2))$. Then we have
\begin{align}\label{eq:W12}
\E\bigg[\sup_{0\leq t\leq T}\int_{\mt}|\bfu(t)|^2\dx+\int_0^{T}\int_{\mt}|\nabla\bfu|^2\dxt\bigg]^{\frac{r}{2}}\leq\,c\,\E\Big[1+\|\bfu_0\|_{L^2_x}^2\Big]^{\frac{r}{2}},
\end{align}
where $\bfu$ is the weak pathwise solution to \eqref{eq:SNS}, cf. Definition \ref{def:inc2d}.
\item[(b)] Assume that $\bfu_0\in L^r(\Omega,W^{1,2}_{0,\Div}(\mt;\R^2))$ for some $r\geq2$ and $\Phi\in L_2(\mathfrak U;W^{1,2}_{0,\Div}(\mt;\R^2))$. Then we have
\begin{align}\label{eq:W22}\begin{aligned}
\E\bigg[\sup_{0\leq t\leq T}\int_{\mt}|\nabla\bfu(t\wedge \mathfrak t_R)|^2\dx&+\int_0^{T\wedge \mathfrak t_R}\int_{\mt}|\nabla^2\bfu|^2\dxt\bigg]^{\frac{r}{2}}\\&\leq\,cR^r\,\E\Big[1+\|\bfu_0\|_{W^{1,2}_x}^2\Big]^{\frac{r}{2}},
\end{aligned}
\end{align}
where $(\bfu,(\mathfrak t_R)_{R\in\N},\mathfrak t)$ is the maximal strong pathwise solution to \eqref{eq:SNS}, cf. Definition \ref{def:maxsol}.
\item[(c)] Assume that $\bfu_0\in L^r(\Omega,W^{2,2}(\mt;\R^2))\cap L^{5r}(\Omega,W^{1,2}_{0,\Div}(\mt;\R^2))$ for some $r\geq2$ and $\Phi\in L_2(\mathfrak U;W^{1,2}_{0,\Div}\cap W^{2,2}(\mt;\R^2))$. Then we have
\begin{align}\label{eq:W32}
\begin{aligned}
\E\bigg[\sup_{0\leq t\leq T}\int_{\mt}|\nabla^2\bfu(t\wedge \mathfrak t_R)|^2\dx&+\int_0^{T\wedge \mathfrak t_R}\int_{\mt}|\nabla^3\bfu|^2\dxt\bigg]^{\frac{r}{2}}\\&\leq\,cR^r\,\E\Big[1+\|\bfu_0\|_{W^{2,2}_x}^2\Big]^{\frac{r}{2}},
\end{aligned}
\end{align}
where $(\bfu,(\mathfrak t_R)_{R\in\N},\mathfrak t)$ is the maximal strong pathwise solution to \eqref{eq:SNS}, cf. Definition \ref{def:maxsol}.
\end{enumerate}
Here $c=c(r,T,\Phi)$ is independent of $R$.
\end{lemma}

\subsection{Estimates for the time-discrete solution}
\label{sec:disest}
We now consider a temporal approximation of  \eqref{eq:SNS} on an equidistant partition of $[0,T]$ with mesh size $\tau=T/M$, and set $t_m=m\tau$. 
Let $\bfu_{0}$ be an $\mathfrak F_0$-measurable random variable with values in $L^{2}_{\Div}(\mt;\R^2)$. For $1 \leq m \leq M$, we aim at constructing iteratively a sequence of $\mathfrak F_{t_m}$-measurable random variables $\bfu_{m}$ with values in $W^{1,2}_{0,\Div}(\mt;\R^2)$ such that
for every $\bfphi\in W^{1,2}_{0,\Div}(\mt;\R^2)$ it holds true $\p$-a.s.
\begin{align}\label{tdiscr}
\begin{aligned}
\int_{\mt}&\bfu_{m}\cdot\bfvarphi \dx -\tau\int_{\mt}\bfu_{m}\otimes\bfu_{m-1} :\nabla\bfphi\dx\\
&\qquad=-\mu\tau\int_{\mt}\nabla\bfu_{m}:\nabla\bfphi\dx+\int_{\mt}\bfu_{m-1}\cdot\bfvarphi \dx+\int_{\mt}\Phi\,\Delta_mW\cdot \bfvarphi\dx,
\end{aligned}
\end{align}
where $\Delta_m W=W(t_m)-W(t_{m-1})$. 
For given $\bfu_{m-1}$ and $\Delta_m W$, to verify the existence o af unique $\bfu_m$  solving \eqref{tdiscr} is straightforward since it is linear in $\bfu_m$.
% In fact, we have $\bfu_m\in W^{2,2}(\mt)$ $\mathbb P$-a.s.~since the stochastic integral belongs to $L^2(\mt)$ $\mathbb P$-a.s. This follows from standard results about the steady Navier--Stokes equations. 
Hence we can rewrite \eqref{tdiscr} as
\begin{align}\label{tdiscr'}
\bfu_{m} +\tau\mathcal P(\nabla\bfu_{m})\bfu_{m-1}=\mu\tau\mathcal A\bfu_m+\bfu_{m-1}+\Phi\,\Delta_mW \qquad (1 \leq m \leq M)\, ,
\end{align}
an identity which holds $\mathbb P$-a.s.~in $L^2(\mt;\R^2)$.
Here $\mathcal P:L^2(\mt;\R^2)\rightarrow L^2_{\Div}(\mt;\R^2)$ denotes the Helmholtz projection and $\mathcal A=\mathcal P\Delta$ is the Stokes operator.

%
%
%Fix $R > 0$; below, we work with the (discrete) $(\mathfrak{F}_{t_m})$-stopping time
%$$\mathfrak t_R^{\tt discr} := \min \bigl\{ t_m \geq 0:\ \int_{0}^{t_m} \vert {\bf u}_m\vert^2
%\Vert \nabla {\bf u}_m\Vert^2 \, \dx \geq R\bigr\} \qquad (R>0)\, ,  $$
%and denote \mathfrak t^{\tt discr} = \lim_{R \rightarrow \infty} \mathfrak t_R^{\tt discr}$ ${\mathbb P}$-a.s.
\begin{lemma}\label{lemma:3.1a} 
Assume that $\bfu_0\in L^{2q}(\Omega,L^{2}_{\Div}(\mt;\R^2))$ for some $q\in\N$. Suppose that $\Phi\in L_2(\mathfrak U;L^2_{\Div}(\mt;\R^2))$. Then the iterates $(\bfu_m)_{m=1}^M$ given by \eqref{tdiscr}
satisfy the following estimate uniformly in $M$:
\begin{align}
\label{lem:3.1a}\E\bigg[\max_{1\leq m\leq M}\|\bfu_m\|^{2q}_{L^{2}_x}+\tau\sum_{m=1}^M\|\bfu_m\|_{L^{2}_x}^{2q-2}\|\nabla\bfu_m\|^2_{L^2_x}\bigg]&\leq\,c,
\end{align}
where $c=c(q,T,\Phi,\bfu_0)>0$. 
\end{lemma}
\begin{proof}
The proof of \eqref{lem:3.1a} is identical to \cite[Lemma 3.1]{BCP}. The latter one is for perodic boundary conditions, but in the case of no-slip boundary conditions the same arguments apply.
Also note that we consider a semi-implicit algorithm, which again does not impact the proof since the convective term still cancels when testing with $\bfu_m$.
 \end{proof}
 For $R_1>0$, we define the (discrete) $(\mathfrak{F}_{t_{m}})$-stopping time
\begin{eqnarray*} 
%{\mathfrak t}_{R_1}^{\tt d} &:=& \min_{0 \leq m \leq M} \bigl\{ t_m:\ \Vert \nabla {\bf u}_m\Vert^2_{L^2_x} \geq R_1\bigr\}\,, \\ 
{\mathfrak s}_{R_1}^{\tt d} &:=& \min_{0 \leq m \leq M} \biggl\{ t_m:\ \sum_{n=0}^m \tau \Vert {\bf u}_{n}\Vert^2_{L^2_x}\Vert \nabla {\bf u}_{n}\Vert^2_{L^2_x} \geq R_1^4\biggr\}\, .
\end{eqnarray*}
We set ${\mathfrak s}_{R_1}^{\tt d}=t_M$ if the set above is empty.
 Note that ${\mathfrak s}_{R_1}^{\tt d} \in \{t_m\}_{m=0}^M$, 
 with random index 
% ${\mathfrak i}_{R_1} \in {\mathbb N}_0 \cap [0,M]$, such that ${\mathfrak t}_{R_1}^{\tt d} = t_{{\mathfrak i}_{R_1}}$; correspondingly 
 ${\mathfrak j}_{R_1} \in {\mathbb N}_0 \cap [0,M]$, such that ${\mathfrak s}_{R_1}^{\tt d} = t_{{\mathfrak j}_{R_1}}$. Now we use the discrete energy estimate \eqref{lem:3.1a} (for $q=2$), and Markov's inequality to validate (assuming that $\bfu_0\in L^{4}(\Omega,W^{1,2}_{0,\Div}(\mt;\R^2))$)
\begin{eqnarray}\nonumber
{\mathbb P}\bigl[ {\mathfrak s}_{R_1}^{\tt d} < T\bigr]
&\leq& {\mathbb P}\biggl[ \sum_{n=0}^M \tau \Vert {\bf u}_n\Vert^2_{L^2_x} \Vert \nabla {\bf u}_n\Vert^2_{L^2_x} \geq R_1^4\biggr] \\ \label{proof1zz}
&\leq& \frac{1}{R_1^4} {\mathbb E}\biggl[ \sum_{n=0}^M \tau \Vert {\bf u}_n\Vert^2_{L^2_x} \Vert \nabla {\bf u}_n\Vert^2_{L^2_x}\biggr]
\leq \frac{c}{R_1^4}\, ,
\end{eqnarray}
 and thus 
 \begin{align}
 \lim_{R_1 \rightarrow \infty} {\mathbb P}\bigl[ {\mathfrak s}_{R_1}^{\tt d} < T\bigr] = 0.
 \end{align}

 Similarly to Lemma \ref{lem:reg} higher order estimates can only be achieved with the help of a stopping time. 
 We now derive some uniform estimates for the solution of the time-discrete problem, which hold up to the discrete stopping time ${\mathfrak s}_{R_1}^{\tt d}$ with $R_1>0$.
%Similarly to Lemma \ref{lem:reg}
%we have to introduce a sample set and sets for $R>0$ and $m\in\{1,\dots,M\}$\footnote{Brauchen wir das fuer allgemeine $q$?}
%\begin{align*}
%\Omega_R^{\tau,q}:=\bigg\{\omega\in\Omega:\tau\sum_{n=1}^M\|\bfu_n\|_{L^2_x}^{2q-2}\|\nabla\bfu_n\|_{L^2_x}^2\leq R\bigg\},\\
%\Omega_{R,m}^{\tau,q}:=\bigg\{\omega\in\Omega:\tau\sum_{n=1}^m\|\bfu_n\|_{L^2_x}^{2q-2}\|\nabla\bfu_n\|_{L^2_x}^2\leq R\bigg\}.
%\end{align*}
%By \eqref{lem:3.1a} we can control the size of $\Omega_R^\tau$, whereas \eqref{lem:3.1b} and \eqref{lem:3.1c} only hold in $\Omega_R^\tau$

\begin{lemma}\label{lemma:3.1} 
Assume that $\bfu_0\in L^{2q}(\Omega,W^{1,2}_{0,\Div}(\mt;\R^2))$ for some $q\in\N$. Suppose that $\Phi\in L_2(\mathfrak U;W^{1,2}_{0,\Div}(\mt;\R^2))$. Then the iterates $(\bfu_m)_{m=1}^M$ given by \eqref{tdiscr}
satisfy the following estimates uniformly in $M$:
\begin{align}
\label{lem:3.1b}\E\bigg[\max_{1\leq m\leq {\mathfrak j}_{R_1} }\|\bfu_m\|^{2q}_{W^{1,2}_x}+\sum_{m=1}^{{\mathfrak j}_{R_1} }\tau\|\bfu_m\|_{W^{1,2}_x}^{2q-2}\|\nabla^2\bfu_m\|^2_{L^2_x}\bigg]&\leq\,ce^{cR^{4}_1},
%\label{lem:3.1c}\E\bigg[\mathbf{1}_{\Omega_{R}^{\tau,q}}\bigg(\sum_{m=1}^M\|\bfu_m-\bfu_{m-1}\|^2_{W^{1,2}_x}\bigg)^4+\mathbf{1}_{\Omega_{R}^{\tau,q}}\bigg(\tau\sum_{m=1}^M\|\nabla^2\bfu_m\|^2_{L^2_x}\bigg)^4\bigg]&\leq\,c,
\end{align}
where $c=c(q,T,\Phi,\bfu_0)>0$ is independent of $R_1$.
\end{lemma}
\begin{proof}
We proceed formally, a rigorous proof can be obtained using a Galerkin approximation.
 In order to prove \eqref{lem:3.1b} we multiply \eqref{tdiscr}
 by $\mathcal A\bfu_m$ and integrate in space yielding
 \begin{align*}
\int_{\mt}\nabla(\bfu_{m}-&\bfu_{m-1}):\nabla\bfu_m\dx +\mu\tau\int_{\mt}|\mathcal A\bfu_m|^2\dx\\&=-\tau\int_{\mt}(\nabla\bfu_m)\bfu_{m-1}\cdot\mathcal A\bfu_m\dx+\int_{\mt}\mathcal A\bfu_m\cdot\Phi\Delta_mW\dx.
\end{align*}
For $\delta>0$ we estimate
\begin{align}\label{eq:1407}
\begin{aligned}
\int_{\mt}(\nabla\bfu_m)\bfu_{m-1}\cdot\mathcal A\bfu_m&\leq\|\bfu_{m-1}\|_{L^4_x}\|\nabla\bfu_m\|_{L^4_x}\|\mathcal A\bfu_m\|_{L^2_x}
\\&\leq\,c\|\bfu_{m-1}\|_{L^2_x}^{\frac{1}{2}}\|\nabla\bfu_{m-1}\|_{L^2_x}^{\frac{1}{2}}\|\nabla\bfu_m\|_{L^2_x}^{\frac{1}{2}}\|\mathcal A\bfu_m\|_{L^2_x}^{\frac{3}{2}}\\&\leq\,c(\delta)\|\bfu_{m-1}\|_{L^2_x}^2\|\nabla\bfu_{m-1}\|_{L^2_x}^2\|\nabla\bfu_m\|_{L^2_x}^2+\delta\|\mathcal A\bfu_m\|_{L^2_x}^{2}.
\end{aligned}
\end{align}
Summing up then shows
 \begin{align*}
\tfrac{1}{2}\int_{\mt}|\nabla\bfu_{m}|^2\dx&+\tfrac{1}{2}\sum_{n=1}^m\int_{\mt}|\nabla(\bfu_{n}-\bfu_{n-1})|^2\dx +\tfrac{\mu}{2} \sum_{n=1}^m\tau\int_{\mt}|\A\bfu_{n}|^2\dx\\&\leq \tfrac{1}{2} \int_{\mt}|\nabla\bfu_{0}|^2\dx+c \sum_{n=1}^m\tau\|\bfu_{n-1}\|_{L^2_x}^2\|\nabla\bfu_{n-1}\|_{L^2_x}^2\|\nabla\bfu_n\|_{L^2_x}^2\\&+\mathscr M^1(t_m)+\mathscr M_m^2,
\end{align*}
where
\begin{align*}
\mathscr M^1(t)&=\int_{\mt}\int_{0}^{t}\sum_{n=1}^M\mathbf1_{[t_{n-1},t_n)}\Phi\,\dd W\cdot \A\bfu_{n-1}\dx,\\
\mathscr M_m^2&=\sum_{n=1}^m\int_{\mt}\int_{t_{n-1}}^{t_n}\Phi\,\dd W\cdot \A(\bfu_n-\bfu_{n-1})\dx.
\end{align*}
%and
%\begin{align*}
%\Omega_{R,m}^\tau:=\bigg\{\omega\in\Omega:\tau\sum_{n=1}^m\|\bfu_n\|_{L^2_x}^{2q-2}\|\nabla\bfu_n\|_{L^2_x}^2\leq R\bigg\}.
%\end{align*}
By the discrete Gronwall lemma we have $\p$-a.s.
 \begin{align*}
\tfrac{1}{2}\max_{1\leq m\leq {\mathfrak j}_{R_1}}\int_{\mt}|\nabla\bfu_{m}|^2\dx&+\tfrac{1}{2}\sum_{n=1}^{{\mathfrak j}_{R_1}}\int_{\mt}|\nabla(\bfu_{n}-\bfu_{n-1})|^2\dx +\tfrac{\mu}{2} \sum_{n=1}^{{\mathfrak j}_{R_1}}\tau\int_{\mt}|\A\bfu_{n}|^2\dx\\&\leq ce^{cR^2_1}\bigg( \int_{\mt}|\nabla\bfu_{0}|^2\dx+\max_{1\leq m\leq {\mathfrak j}_{R_1}}|\mathscr M^1(t_m)|+\max_{1\leq m\leq {\mathfrak j}_{R_1}}|\mathscr M_m^2|\bigg)
\end{align*}
using that
\begin{align*}
\sum_{n=1}^{{\mathfrak j}_{R_1}}\tau\|\bfu_{n-1}\|_{L^2_x}^2\|\nabla\bfu_{n-1}\|_{L^2_x}^2\leq R_1^4
\end{align*}
by the definition of  ${\mathfrak j}_{R_1}$.
Since $\bfu_{m-1}$ is $\mathfrak F_{t_{m-1}}$-measurable and $\mathfrak s_{R_1}^{\tt d}$ is an $(\mathfrak F_{t_{m}})$-stopping time we know that $\mathscr M^1(t\wedge\mathfrak s_{R_1}^{\tt d})$ is an $(\mathfrak F_t)$-martingale. Consequently,
 by Burkholder-Davis-Gundy inequality, $\Phi\in L_2(\mathfrak U;W^{1,2}_{0,\Div}(\mt;\R^2))$ and Young's inequality, and for $\kappa>0$
\begin{align}\nonumber
\E\bigg[\max_{1\leq m\leq {\mathfrak j}_{R_1}}\big|\mathscr M^1(t_m)\big|\bigg]&\leq \E\bigg[\max_{s\in[0,T]}\big|\mathscr M^1(s\wedge\mathfrak s_{R_1}^d)\big|\bigg]\\ \nonumber
&\leq\,c\,\E\bigg[\bigg(\int_{0}^{T\wedge\mathfrak s_{R_1}^{\tt d}}\sum_{n=1}^M\mathbf1_{[t_{n-1},t_n)}\|\Phi\|^2_{L_2(\mathfrak U,L^2_x)}\|\A\bfu_{n-1}\|^2_{L^2_x}\ds\bigg)^{\frac{1}{2}}\bigg]\\ \label{term1}
&\leq\,c\,\E\bigg[\bigg(\sum_{n=1}^{{\mathfrak j}_{R_1}-1}\|\A\bfu_{n}\|^2_{L^2_x}\bigg)^{\frac{1}{2}}\bigg]\\ \nonumber
&\leq\,c(\kappa)+\,\kappa  \E\bigg[\sum_{n=1}^{{\mathfrak j}_{R_1}-1}\tau\|\A\bfu_{n}\|_{L^2_x}^2\bigg].
\end{align} 
Furthermore, we have
\begin{align}\nonumber
\E\bigg[\max_{1\leq m\leq {\mathfrak j}_{R_1}}|\mathscr M^2_{m}|\bigg]&\leq \,\kappa\,\E\bigg[  \sum_{n=1}^{{\mathfrak j}_{R_1}}\big\|\nabla(\bfu_{n}-\bfu_{n-1})\big\|_{L^2_x}^2\bigg]+c_\kappa\,\E\bigg[\sum_{n=1}^{{\mathfrak j}_{R_1}}\bigg\| \int_{t_{n-1}}^{t_n}\nabla\Phi\,\dd W  \bigg\|_{L^2_x}^2\bigg]\\ \nonumber
&\leq \,\kappa\,\E\bigg[ \sum_{n=1}^{{\mathfrak j}_{R_1}}\|\nabla(\bfu_{n}-\bfu_{n-1})\|_{L^2_x}^2 \bigg]+ c_\kappa\,\E\bigg[\tau\sum_{n=1}^M\|\Phi\|_{L_2(\mathfrak U;W^{1,2}_x)}^2\bigg]\\ \label{term2}
&\leq \,\kappa\,\E\bigg[\sum_{n=1}^{{\mathfrak j}_{R_1}} \|\nabla(\bfu_{n}-\bfu_{n-1})\|_{L^2_x}^2 \bigg]+ c_\kappa
\end{align}
due to Young's inequality, It\^{o}-isometry and $\Phi\in L_2(\mathfrak U;W^{1,2}_{0,\Div}(\mt;\R^2))$. Absorbing the $\kappa$-terms we conclude
(b) for $q=1$. The case $q\geq 2$ follows similarly by multiplying
with $\|\bfu_m\|_{W^{1,2}_x}^{2q-2}$ and iterating (see \cite[Lemma 3.1]{BCP} for details).
\end{proof}

 For $R_2>0$, we define the (discrete) $(\mathfrak{F}_{t_{m}})$-stopping time
\begin{eqnarray*} {\mathfrak t}_{R_2}^{\tt d} &:=& \min_{0 \leq m \leq M} \biggl\{ t_m:\ \Vert \nabla {\bf u}_m\Vert^2_{L^2_x}+\sum_{n=1}^{m}\tau\|\nabla^2\bfu_n\|^2_{L^2_x}\ \geq R_2^2\biggr\}\,,
\end{eqnarray*}
and set ${\mathfrak t}_{R_2}^{\tt d}=t_M$ if the set above is empty.
 Note that 
${\mathfrak t}_{R_2}^{\tt d} \in \{t_m\}_{m=0}^M$, with random index ${\mathfrak i}_{R_2} \in {\mathbb N}_0 \cap [0,M]$, such that ${\mathfrak t}_{R_2}^{\tt d} = t_{{\mathfrak i}_{R_2}}$. The crucial observation is now
that
\begin{equation}\label{lemma:3.1c}
\lim_{R_1 \rightarrow \infty}{\mathbb P}\bigl[{\mathfrak t}_{R_2(R_1)}^{\tt d} < T\bigr] = 0\, ,
\end{equation}
provided we choose $R_2=R_2(R_1)$ such that $e^{cR_1^4}=o(R_2(R_1)^2)$.
We argue as in \cite[Sec.~4.3]{GlZi} to show (\ref{lemma:3.1c}) and estimate
 \begin{equation}\label{proof1zzz} {\mathbb P}\bigl[{\mathfrak t}_{R_2}^{\tt d} < T\bigr] \leq
 {\mathbb P} \bigl[ \{ {\mathfrak t}_{R_2}^{\tt d} < T\}\cap \{ {\mathfrak s}_{R_1}^{\tt d} \geq T\} \bigr] + {\mathbb P}\bigl[{\mathfrak s}_{R_1}^{\tt d} < T\bigr]\, .
 \end{equation}
 By the notation introduced before this lemma, the first term on the right-hand side of (\ref{proof1zzz}) can thus be bounded by
 \begin{align}\label{proof1zzy}
 \begin{aligned}
 {\mathbb P}\biggl[\biggl\{\Vert \nabla {\bf u}_{{\mathfrak i}_{R_2}}\Vert^2_{L^{2}_x}&+\sum_{n=1}^{{\mathfrak i}_{R_2}}\tau\|\nabla^2\bfu_n\|^2_{L^2_x} \geq R_2^2 \biggr\} \cap \{ {\mathfrak s}_{R_1}^{\tt d} \geq T\}\biggr]\\
& \leq {\mathbb P}\biggl[ \biggl\{\Vert \nabla {\bf u}_{{\mathfrak i}_{R_2} \wedge {\mathfrak j}_{R_1}}\Vert^2_{L^{2}_x} +\sum_{n=1}^{{\mathfrak i}_{R_2}\wedge \mathfrak j_{R_1}}\tau\|\nabla^2\bfu_n\|^2_{L^2_x}\geq R_2^2 \biggr\}\biggr]\, .
 \end{aligned}
 \end{align}
 By Markov's inequality, the right-hand side is bounded by
$$\frac{1}{R_2^2} {\mathbb E}\biggl[ \Vert \nabla {\bf u}_{{\mathfrak i}_{R_2} \wedge {\mathfrak j}_{R_1}}\Vert^2_{L^{2}_x}+\sum_{n=1}^{{\mathfrak i}_{R_2}\wedge \mathfrak j_{R_1}}\tau\|\nabla^2\bfu_n\|^2_{L^2_x}\biggr]\, .$$
%
%\leq {\mathbb E}\bigl[ \sup_{0 \leq m \leq M}\Vert \nabla {\bf u}_m\Vert^2_{L^{2}_x}\bigr]; $$ 
Using Lemma \ref{lemma:3.1} and choosing $R_2$ such that $e^{cR_1^4}=o(R_2^2)$, we may deduce  
\begin{equation}\label{proof1zzx}
\lim_{R_2 \rightarrow \infty} \frac{1}{R_2^2} {\mathbb E}\biggl[ \Vert \nabla {\bf u}_{{\mathfrak i}_{R_2} \wedge {\mathfrak j}_{R_1}}\Vert^2_{L^{2}_x}+\sum_{n=1}^{{\mathfrak i}_{R_2} \wedge \mathfrak j_{R_1}}\tau\|\nabla^2\bfu_n\|^2_{L^2_x}\biggr] = 0\,,
\end{equation}
 and then (\ref{lemma:3.1c}) follows  from (\ref{proof1zzz}).\\\

%Setting
%$\tau_{m}^R:=\mathfrak t_{m}^R-\mathfrak t_{m-1}^R$
%we introduce $\bfu_{m}^R$ as the solution to
%\begin{align}\label{tdiscrR}
%\begin{aligned}
%\int_{\mt}&\bfu_{m}^R\cdot\bfvarphi \dx -\tau_{m}^R\int_{\mt}\bfu_{m}^R\otimes\bfu^R_{m-1}:\nabla\bfphi\dx\\
%&=\mu\,\tau_{m}^R\int_{\mt}\nabla\bfu_{m}^R:\nabla\bfphi\dx=\int_{\mt}\bfu_{m-1}^R\cdot\bfvarphi \dx+\frac{\tau_{m}^R}{\tau}\int_{\mt}\Phi\,\Delta_m W\cdot \bfvarphi\dx,
%\end{aligned}
%\end{align}
%for every $\bfphi\in W^{1,2}_{0,\Div}(\mt)$. Obviously $\bfu_{m}^R=\bfu_{m}$ in $[t_m\leq  \mathfrak t_{m}^R]$.
 The equation (\ref{tdiscr}) has been stated for solenoidal test functions
$\bfphi\in W^{1,2}_{0,\Div}(\mt;\R^2)$; for $\bfphi\in  W^{1,2}_{0}(\mt;\R^2)$, the additional term
$$ -{\tau} \int_{\mathcal O} p_m {\rm div}\, \bfphi\,{\rm d}x $$
would appear on the left-hand side of (\ref{tdiscr}). The ${\mathbb P}$-a.s.~unique solvability for
$({\bf u}_m, p_m) \in W^{1,2}_{0, {\rm div}}(\mt;\R^2) \times L^2(\mt)/{\mathbb R}$ follows from standard arguments. Moreover, exploiting that $\Phi\in L_2(\mathfrak U;L^2_{\Div}(\mt;\R^2))$, we easily deduce
\begin{equation}\label{tdiscrRp}
\E\bigg[ \sum_{m=1}^{\mathfrak{j}_{R_1}} \tau \|\nabla p_m\|^2_{L^2_x} \bigg]\leq\,ce^{cR_1^4},
\end{equation}
from (\ref{lem:3.1b}),
where $c=c(q,T,\Phi,\bfu_0)>0$ is independent of $R_1$: therefore, we start from (\ref{tdiscr}) $\mathbb P$-a.s. in strong form,
\begin{equation}\label{tdiscrRp1} {\bf u}_m - \mu \tau \Delta {\bf u}_m + \tau {\rm div} \bigl({\bf u}_m \otimes {\bf u}_{m-1}\bigr) + \tau \nabla p_m = 
{\bf u}_{m-1} + \Phi \Delta_m W\, ,
 \end{equation}
 which we multiply with $\nabla p_m$ and integrate in space; a standard calculation then leads to
 \begin{equation*}
 \tau \Vert \nabla p_m\Vert^2 \leq C\tau \bigl( \Vert \Delta {\bf u}_m\Vert^2 + \Vert {\bf u}_{m-1}\Vert_{L^2_x}  \Vert \nabla{\bf u}_{m-1}\Vert_{L^2_x} \Vert \nabla {\bf u}_m\Vert_{L^2_x} \Vert \Delta {\bf u}_m\Vert_{L^2_x}\bigr) + \frac{\tau}{2} \Vert \nabla p_m\Vert^2\,.
 \end{equation*}
Hence (\ref{lem:3.1b}) settles (\ref{tdiscrRp}).
\begin{remark}\label{rem}
\begin{enumerate}
\item It is possible to prove Lemma \ref{lemma:3.1} also for nonlinear multiplicative noise provided suitable growth conditions for the derivatives of the diffusion coefficient with respect to the velocity field are assumed.
\item The assumption that the noise vanishes at the boundary is crucial for the estimates in Lemma \ref{lemma:3.1}. Otherwise, we are unable to integrate by parts in $\mathscr M_m^2$ below \eqref{eq:1407}, and it is unclear how to proceed. Interestingly, this limitation does not occur in the time-continuous counterpart of Lemma \ref{lemma:3.1}, see \cite[Lemma 3.1]{BrPr}.
However, it is not needed in Lemma \ref{lemma:3.1} that the noise is solenoidal; this assumption is only needed for the estimate \eqref{tdiscrRp} and will become crucial in the next section.
\item Without the vanishing trace condition for $\Phi$ we are currently unable to control the error between the temporal discretisation and the spatio-temporal discretisation. On account of this we studied in \cite{BrPr} directly the error between the exact solution and its space-time approximation. 
\end{enumerate}
\end{remark}

\section{Asymptotic strong first order error bounds for the time discretisation}\label{sec:add}
 If $\bfu$ is the weak pathwise solution defined on $(\Omega,\mf,(\mf_t)_{t\geq0},\prst)$ with an $(\mf_t)$-cylindrical Wiener process $W$ (recall Definition \ref{def:inc2d} and Theorem \ref{thm:inc2d}), we may consider the transform
\begin{equation}\label{transform} {\bf y}(t) = {\bf u}(t) - \int_0^t \Phi \, {\rm d}W(s)=\bfu(t)-\Phi W(t) \quad \forall\, t \geq 0.
\end{equation}
We denote
\begin{align*}
\mathcal L_1^W(\bfy)&:= 
 \bigl( \nabla[\Phi W] \bigr){\bf y}  ,\quad
\mathcal L_2^W(\bfy):= 
\bigl(\nabla {\bf y}\bigr) [\Phi W]  ,\\
\mathcal L_3^W(t) &:=  \bigl(\nabla [\Phi W]\bigr)[\Phi W],\quad \mathcal L^W:=\mathcal L_1^W+\mathcal L_2^W+\mathcal L^W_3 .
 \end{align*}
Then ${\bf y}: [0,T] \times {\mathcal O} \times \Omega \rightarrow {\mathbb R}^2$ solves the random PDE
\begin{align}\label{eq:SNSy}
\left\{\begin{array}{rc}
\partial_t {\bf y}  =\mu {\mathcal A} {\bf y}-{\mathcal P} \bigl[(\nabla {\bf y}){\bf y}\bigr]  +\mu {\mathcal A} [\Phi W]
-{\mathcal P} \bigl[\mathcal L^W(\bfy)\bigr]
& \mbox{in $\mathcal Q_T$,}\\
\Div {\bf y}=0\qquad\qquad\qquad\qquad\qquad\,\,\,\,& \mbox{in $\mathcal Q_T$,}\\
{\bf y}(0)=\bfu_0\,\qquad\qquad\qquad\qquad\qquad&\mbox{ \,in $\mt$.}\end{array}\right.
\end{align}
Note that for $\mathbb P$-a.a. $\omega\in\Omega$
the function $\bfy(\omega,\cdot)$ is a solution to the Navier--Stokes equations with right-hand side
\begin{align*}
\bff:=\mu {\mathcal A} [\Phi W]
-{\mathcal P} \bigl[\mathcal L^W(\bfy)\bigr].
\end{align*} 
Standard regularity results apply provided $\Phi$ is sufficiently regular.
In particular, $\partial_t {\bf y} \in L^2\bigl( 0,T; L^2_{\rm div}({\mathcal O};\R^2) \bigr)$ holds ${\mathbb P}$-a.s.
\begin{lemma}\label{lem:regadditive}
Suppose $\bfu_0\in L^r(\Omega;L^2_{\Div}(\mt;\R^2))$ for some $r>2$ and $\Phi\in L_2(\mathfrak U;W^{1,2}_{0,\Div}(\mt;\R^2))$. Let $\bfu$ be the unique weak pathwise solution to \eqref{eq:SNS}.
\begin{enumerate} 
\item[(a)] Assume additionally $\bfu_0\in W^{1,2}_{0,\Div}(\mt;\R^2)$ $\mathbb P$-a.s.~and $\Phi\in L_2(\mathfrak U;W^{2,2}(\mt;\R^2))$. Then $\partial_t {\bf y} \in L^2\bigl( 0,T; L^2_{\rm div}({\mt};\R^2) \bigr)$  $\mathbb P$-a.s.~and for a.a. $t\in(0,T)$
\begin{align}\label{eq:W22y}
\int_{\mt}|\partial_t\bfy|^2\dx\leq\,c\,\Big[\|\Phi W\|_{W^{2,2}_x}^2+ \|\bfu\|_{W^{2,2}_x}^2+\|\Phi W\|_{W^{1,2}_x}^4+\|\bfu\|_{W^{1,2}_x}^4\Big].
\end{align}
\item[(b)] Assume additionally $\bfu_0\in W^{2,2}\cap W^{1,2}_{0,\Div}(\mt;\R^2)$ $\mathbb P$-a.s.~and $\Phi\in L_2(\mathfrak U;W^{3,2}(\mt;\R^2))$. Then $\partial_t {\bf y} \in L^2\bigl( 0,T; W^{1,2}_{\rm div}({\mt};\R^2) \bigr)$  $\mathbb P$-a.s.~and
\begin{align}\label{eq:W32y}
\int_{\mathcal Q_T}|\partial_t\nabla\bfy|^2\dxt\leq\,c\,\int_0^T\Big[\|\Phi W\|_{W^{3,2}_x}^2+\|\bfu\|_{W^{3,2}_x}^2+\|\Phi W\|_{W^{2,2}_x}^4+\|\bfu\|_{W^{2,2}_x}^4\Big]\dt.
\end{align}
\end{enumerate}
\end{lemma}
\begin{proof}
The proof follows directly from \eqref{eq:SNSy}$_1$ by estimating the right-hand side in $L^2_x$ and $W^{1,2}_x$ respectively. This only uses Ladyshenskaya's inequality to estimate the quadratic terms.
Note that our assumptions imply sufficient regularity of $\bfu$ by Lemma \ref{lem:reg}.
\end{proof}
Combining Lemmas \ref{lem:regadditive} and \ref{lem:reg} we obtain the following.
\begin{corollary}\label{cor:add}
Assume that $\bfu_0\in L^r(\Omega,W^{2,2}_{\Div}(\mt;\R^2))\cap L^{5r}(\Omega,W^{1,2}_{0,\Div}(\mt;\R^2))$ for some $r\geq2$ and that $\Phi\in L_2(\mathfrak U;W^{3,2}(\mt;\R^2))\cap L_2(\mathfrak U;W^{1,2}_{0,\Div}(\mt;\R^2))$. Let $\bfu$ be the unique weak pathwise solution to \eqref{eq:SNS}.
Then we have for any $R>0$
\begin{align*}
\E\bigg[\Bigl(\sup_{0\leq t\leq T}\int_{\mt}|\partial_t\bfy(t\wedge\mathfrak t_R)|^2\dx\Bigr)^{\frac{r}{2}}\bigg]&\leq c(T,\Phi,\bfu_0)R^{2r},\\\E\bigg[\Bigl(\int_0^{\mathfrak t_R}\int_{\mathcal O}|\partial_t\nabla\bfy|^2\dxt\Bigr)^{\frac{r}{2}}\bigg]&\leq\,c(T,\Phi,\bfu_0)R^{2r},\\ \E\bigg[\Bigl(\sup_{0\leq t\leq T}\int_{\mt}|\nabla\bfy(t\wedge\mathfrak t_R)|^2\dx\Bigr)^{\frac{r}{2}}\bigg]&\leq c(T,\Phi,\bfu_0)R^{r}.
\end{align*}
\end{corollary}

\subsection{A stable time-discretisation}
\label{sec:3.1}
Due to the time-regularity of $\bfy$ stated in Lemma \ref{lem:regadditive} above we expect strong order (up to) $1$ in probability for the semi-implicit temporal discretisation of \eqref{eq:SNS}. We consider an equidistant partition of $[0,T]$ with mesh size $\tau=T/M$ and set $t_m=m\tau$. 
Let $\bfu_{0}$ be an $\mathfrak F_0$-measurable random variable with values in $W^{1,2}_{0,\Div}(\mt;\R^2)$. Furthermore, we assume that  $\Phi\in L_2(\mathfrak U;W^{1,2}_{0,\Div}(\mt;\R^2))$.
Given $\bfy_0=\bfu_0$ we seek for an $\mf_{t_m}$-measurable, $W^{1,2}_{0,\rm div}(\mt;\R^2)$-valued random variable ${\bf y}_m$ ($1 \leq m \leq M$) such that
%$(\bfu_{m})$ with values in $W^{1,2}_{\Div}(\mt)$ such that
%for every $\bfphi\in W^{1,2}_{\Div}(\mt)$ it holds true $\p$-a.s.\footnote{voll explicit verwenden wir die Lipschitz NL --- andere, semi-implizite art moeglich?}
\begin{align}\label{tdiscr-add}
\begin{aligned}
\frac{{\bf y}_{m} - {\bf y}_{m-1}}{\tau} - \mu {\mathcal A} {\bf y}_m &= \mu {\mathcal A}[\Phi W(t_m)] - {\mathcal P}\bigl[ (\nabla {\bf y}_{m}){\bf y}_{m-1} \bigr]+{\mathcal P}\bigl[ \mathcal L^{m}(\bfy_{m-1},\bfy_m)\bigr],\\
\text{where}\quad \mathcal L^{m}(\bfy_{m-1},\bfy_m)&=\mathcal L_1^{W(t_m)}(\bfy_{m-1})+\mathcal L_2^{W(t_{m-1})}(\bfy_{m})+\mathcal L_1^{W(t_m)}(\Phi W(t_{m-1})).
\end{aligned}
\end{align}
%for ${\bf y}_m$  where 
%$$I_m + II_m+III_{m} :=  \bigl( \nabla[\Phi W(t_m)]\bigr){\bf y}_{m} + (\nabla {\bf y}_m) [\Phi W(t_{m})] 
%-\bigl( \nabla[\Phi W(t_{m})]\bigr) [\Phi W(t_{m})]$$
%on an equi-distant mesh $\{ t_m\}_{m=0}^M \subset [0,T]$ of size $\tau>0$, and ${\bf y}_0 = {\bf u}_0$. 
This system can be written (for $\mathbb P$-a.a.~$\omega\in\Omega$) as steady Navier--Stokes problem with
forcing $\mu {\mathcal A}[\Phi W(t_m)] $ perturbed by the linear term ${\mathcal P}\bigl[ \mathcal L^{m}(\bfy_{m-1},\bfy_m)\bigr]$. Clearly we have a continuous dependence on $\bfy_{m-1}$, $W(t_{m-1})$ and $W(t_m)$ which yields the correct measurability. If also $\Phi\in L_2(\mathfrak U;W^{2,2}(\mt;\R^2))$, we have more regularity and it holds $\bfy_m\in W^{2,2}(\mt;\R^2)$ $\mathbb P$-a.s. (and similarly for $W^{3,2}(\mt;\R^2)$ instead of $W^{2,2}(\mt;\R^2)$).
Setting 
\begin{equation}\label{identd1}
{\bf u}_m := {\bf y}_m  +\Phi W(t_m) \qquad (1 \leq m \leq M)\,,
\end{equation}
accordingly gives with $\Delta_m W := W(t_m) - W(t_{m-1})$
\begin{equation}\label{help-1}({\bf u}_m - {\bf u}_{m-1}\bigr) - \mu \tau {\mathcal A}{\bf u}_m + \tau {\mathcal P}(\nabla{\bf u}_m ){\bf u}_{m-1} = \Phi\Delta_m W \qquad (1 \leq m \leq M)
\end{equation}
and ${\bf u}_0 = {\bf y}_0$, 
which has been considered in Section \ref{sec:disest}.
Inequality \eqref{i} of the following lemma is based on corresponding stability bounds for 
$(\bfu_m)_{m=1}^M$ from Lemma \ref{lemma:3.1}. We define
\begin{eqnarray}\label{eq:tRd} 
%{\mathfrak t}_{R_1}^{\tt d} &:=& \min_{0 \leq m \leq M} \bigl\{ t_m:\ \Vert \nabla {\bf u}_m\Vert^2_{L^2_x} \geq R_1\bigr\}\,, \\ 
\tilde{\mathfrak t}_{R_1}^{\tt d} &:=& \min_{0 \leq m \leq M} \biggl\{ t_m\leq \mathfrak t_{R_1}:\ \sup_{t\in[0,t_m]}\|\Phi (W_{t})\|_{W^{2,2}_x} \geq R_2(R_1)\biggr\}\, \wedge {\mathfrak s}_{R_1}^{\tt d}\wedge {\mathfrak t}_{R_2(R_1)}^{\tt d} ,
\end{eqnarray}
where the minimum of the empty set is defined as $t_M$ and $R_2(R_1)$ is chosen in accordance with \eqref{eq:2008} below. 
Finally, $\mathfrak m_{R_1}$ denotes the unique index in $\{1,\dots,M\}$ with
$t_{\mathfrak m_{R_1}}=\tilde{\mathfrak t}_{R_1}^{\tt d}$. 
Since $\Phi\in L_2(\mathfrak U;W^{3,2}(\mt;\R^2))$ we clearly have
$$\E\bigg[\sup_{0\leq t\leq t_M}\|\Phi W(t)\|_{W^{2,2}_x}\bigg]\leq\,c.$$
%by the embedding $W^{2,2}_x\hookrightarrow W^{1,4}_x$.
Hence we can control the size of $\{\tilde{\mathfrak t}_{R_1}^{\tt d} < T\} $ by
 \begin{align}\label{eq:2008} \begin{aligned}{\mathbb P}\bigl[\{\tilde{\mathfrak t}_{R_1}^{\tt d} < T\}\bigr] &\leq
 {\mathbb P} \bigl[ \{ {\mathfrak t}_{R_2(R_1)}^{\tt d} < T\}\big]\\&+  {\mathbb P} \bigl[ \{ {\mathfrak s}_{R_1}^{\tt d} < T\}\big]+\mathbb P\bigg[\sup_{0\leq t\leq t_M}\|\Phi W(t)\|_{W^{2,2}_x}\geq R_2(R_1)\bigg]\\
 &\rightarrow 0,
 \end{aligned}
 \end{align}
%\begin{align}\begin{aligned}
%\p&\big([t_M> \tilde{\mathfrak{t}}_M^R]\big)\\&= \p\Big(\Big\{\max_{0\leq t\leq t_M}\|\Phi W(t)\|_{W^{1,4}_x}^{2}+\max_{1\leq n\leq M}\|\bfy_n\|_{W^{1,2}_x}^{2}+\sum_{n=1}^M\tau_n^R\|\bfy_n\|_{W^{2,2}_x}^{2}>K(R)^{2}\Big\}\Big)\\&\leq\frac{c(e^{cR^{4}}+1)}{K(R)^{2}}\rightarrow0,
%\end{aligned}
%\label{eq:2008}
%\end{align}
provided we choose $R_2$ such that $e^{cR_1^4}=o(R_2(R_1)^2)$ (recall \eqref{lemma:3.1c}). Note that we have $\mathbb P$-a.s.
\begin{align*}
\sup_{0\leq t\leq t_{\mathfrak m_{\tiny{R_1}}}}\Big(\|\Phi W(t)\|_{W^{2,2}_x}^{2}&+\|\bfu(t)\|_{W^{1,2}_x}^{2}+\|\bfy(t)\|_{W^{1,2}_x}^{2}\Big)\\&+\max_{1\leq n\leq \mathfrak m_{R_1}-1}\Vert \nabla {\bf u}_{n}\Vert^2_{L^2_x}+\sum_{n=0}^{\mathfrak m_{R_1}-1} \tau\Vert \nabla^2 {\bf u}_{n}\Vert^2_{L^2_x}\\&+\max_{1\leq n\leq \mathfrak m_{R_1}-1}\Vert \nabla {\bf y}_{n}\Vert^2_{L^2_x}+\sum_{n=0}^{\mathfrak m_{R_1}-1} \tau\Vert \nabla^2 {\bf y}_{n}\Vert^2_{L^2_x}\leq R_2(R_1)^2 .
\end{align*}
%$\widetilde\Omega_R^{\tau,q}$ can be controlled by means of the size of $\Omega_R^{\tau,q}$. Finally, we set
%\begin{align*}
%\widetilde\Omega_R^{\tau,q,\ast}&:=\bigg\{\omega\in\widetilde\Omega_R^{\tau,q}:\tau\sum_{n=1}^M\mathbf{1}_{\widetilde\Omega_{R,n-1}^{\tau,q}}\|\nabla\bfy_n\|_{L^2_x}^{2q-2}\|\A\bfy_n\|_{L^2_x}^2\leq R\bigg\},\\
%\widetilde\Omega_{R,m}^{\tau,q,\ast}&:=\bigg\{\omega\in \widetilde\Omega_{R,m}^{\tau,q}:\tau\sum_{n=1}^m\mathbf{1}_{\widetilde\Omega_{R,n-1}^{\tau,q}}\|\nabla\bfy_n\|_{L^2_x}^{2q-2}\|\A\bfy_n\|_{L^2_x}^2\leq R\bigg\}.
%\end{align*}
%The size of $\widetilde\Omega_R^{\tau,q,\ast}$ can be controlled by below \eqref{i} below.
\begin{lemma}\label{lem-rand1}
Assume that $\bfu_0\in L^{2q}(\Omega,W^{3,2}(\mt;\R^2))\cap L^{2q+2}(\Omega, W^{1,2}_{0,\Div}(\Omega;\R^2))$ for some $q\in\N$ and that $\Phi\in L_2(\mathfrak U;W^{3,2}(\mt;\R^2))\cap L_2(\mathfrak U;W^{1,2}_{0,\Div}(\mt;\R^2))$. Then the iterates $(\bfy_m)_{m=1}^M$ given by \eqref{tdiscr-add}
satisfy the following estimates uniformly in $M$:
\begin{align}
\label{i}& {\mathbb E}\biggl[ \max_{1\leq m\leq {\mathfrak m}_{R_1}} \Vert \bfy_m\Vert^{2q}_{W^{1,2}_x} +\sum_{m=1}^{\mathfrak m_{R_1}}\tau \Vert \nabla \bfy_m\Vert^{2^q-2}_{L^2_x} \Vert {\mathcal A} \bfy_m\Vert_{L^2_x}^2\biggr]\leq c\,e^{cR_2(R_1)^{4}} ,\\
\label{ii}& {\mathbb E}\biggl[ \max_{1\leq m\leq {\mathfrak m}_{R_1}} \Vert {\bf y}_m\Vert^{2q}_{W^{2,2}_x}+ \sum_{m=1}^{\mathfrak m_{R_1}} \tau \Vert {\mathcal A} \bfy_m\Vert^{2^q-2}_{L^2_x} \Vert \nabla {\mathcal A} \bfy_m\Vert^2_{L^2_x}\biggr]\leq c\,e^{cR_2(R_1)^{4}},\\
\label{iii} &  {\mathbb E}\biggl[\sum_{m=1}^{\mathfrak m_{R_1}} \tau\Vert \bfy_m\Vert^{2^{q}-4}_{W^{2,2}_x}\bigg\Vert \frac{\bfy_m - \bfy_{m-1}}{\tau}\bigg\Vert^{2}_{W^{1,2}_x}\biggr]
%+\E\bigg[\sum_{m=1}^M\|\bfy_m-\bfy_{m-1}\|_{W^{1,2}_x}^2\bigg] 
\leq ce^{cR_2(R_1)^{4}}\quad (\text{for }q\geq2),
\end{align}
where $c=c(q,T,\bfu_0)>0$ is independent of $R_1$.
\end{lemma}
\begin{proof}
{\bf 1.} Ad \eqref{i}. The proof is similar to the corresponding estimate \eqref{lem:3.1b} given in Lemma \ref{lemma:3.1}.
% in \cite[Lemma 3.1]{BCP} which in particular gives 
%\begin{align}\label{Lemma3.1BCP}
%{\mathbb E}\biggl[ \max_{1\leq m\leq M} \Vert \bfu_m\Vert^{2q}_{W^{1,2}_x}\biggr] 
%+ \tau {\mathbb E} \biggl[ \sum_{m=1}^M  \Vert \nabla {\bf u}^m\Vert^{2^q-2}_{L^2_x} \Vert {\mathcal A} {\bf u}^m\Vert_{L^2_x}^2\biggr]\leq C_q
%\end{align}
%for any $q\in\N$ provided $\bfu_0\in L^{2q}(\Omega;W^{1,2}_{0,\Div}(\mt))$.
Equation 
 \eqref{tdiscr-add} can be rewritten as (see also equation \eqref{eq:SNSy})
 \begin{align*}
\bfy_{m} - {\bf y}_{m-1} - \mu \tau{\mathcal A} {\bf y}_m = \mu\tau {\mathcal A}[\Phi W(t_m)] - \tau{\mathcal P}\bigl[ (\nabla {\bf y}_{m}){\bf y}_{m-1} \bigr]+\tau{\mathcal P}\bigl[ \mathcal L^{m}({\bfy}_{m-1},{\bfy}_m)\bigr].
%\mathcal L^{m,R}(\tilde{\bfy}^R_{m-1},\tilde{\bfy}^R_m&=\mathcal L_1^{W(\tilde{\mathfrak t}_m^R)}(\bfy_{m-1})+\mathcal L_1^{W(\tilde{\mathfrak t}_{m-1})}(\bfy_{m})+\mathcal L_1^{W(\tilde{\mathfrak t}_m)}(\Phi W(\tilde{\mathfrak t}_{m-1})).
\end{align*}
Multiplying by $\mathcal A \bfy_{m}$ and integrating in space yields
 \begin{align*}
\int_{\mt}&\nabla(\bfy_{m}-\bfy_{m-1}):\nabla\bfy_m\dx +\mu\tau\int_{\mt}|\mathcal A\bfy_m|^2\dx\\&=\tau\int_{\mt}(\nabla\bfy_m){\bfy}_{m-1}\cdot\mathcal A\bfy_m\dx
-
\mu\tau\int_{\mt}{\mathcal A}[\Phi W(t_m)]\cdot \mathcal A\bfy\dx 
- \tau\int_{\mt}\mathcal L^{m}({\bfy}_{m-1},\bfy_m)\cdot \mathcal A\bfy_m\dx.
\end{align*}
As in \eqref{eq:1407} we have for $\delta>0$
\begin{align*}
\int_{\mt}(\nabla\bfy_m){\bfy}_{m-1}\cdot\mathcal A\bfy_m\dx&\leq\,c(\delta)\|{\bfy}_{m-1}\|_{L^2_x}^2\|\nabla{\bfy}_{m-1}\|_{L^2_x}^2\|\nabla\bfy_m\|_{L^2_x}^2+\delta\|\mathcal A\bfy_m\|_{L^2_x}^{2}
\end{align*}
such that for $\delta$ sufficiently small
 \begin{align}
\nonumber
\tfrac{1}{2}\int_{\mt}|\nabla{\bfy}_{m}|^2\dx&+\tfrac{1}{2}\int_{\mt}|\nabla({\bfy}_{m}-{\bfy}_{m-1})|^2\dx +\tfrac{\mu}{2}\tau\int_{\mt}|\mathcal A{\bfy}_{m}|^2\dx\\
&\leq  
\tfrac{1}{2}\int_{\mt}|\nabla{\bfy}_{m-1}|^2\dx+c \sum_{n=1}^m\tau\|\bfy_{n-1}\|_{L^2_x}^2\|\nabla\bfy_{n-1}\|_{L^2_x}^2\|\nabla\bfy_n\|_{L^2_x}^2\label{eq:2605}\\&+c\,\tau\int_{\mt}|{\mathcal A}[\Phi W(t_m)]|^2\dx+c\,\tau\,\int_{\mt}|\mathcal L^{m}({\bfy}_{m-1},\bfy_m)|^2\dx.\nonumber
\end{align}
Since  $\Phi\in L_2(\mathfrak U;W^{2,2}(\mt))$ we have
\begin{align*}
\E\Bigl[\tau\sum_{m=1}^{M}\int_{\mt}|{\mathcal A}[\Phi W(t_m)]|^2\dx\Bigr]\leq\,c,
\end{align*}
and choosing $m=\mathfrak m_{R_1}$ yields
\begin{align*}
&\E\Bigl[\sum_{m=1}^{\mathfrak m_{R_1}}\tau\int_{\mt}|\mathcal L^{m}(\bfy_{m-1},\bfy_m)|^2\dx\Bigr]\\&\leq\E\bigg[\sum_{m=1}^{\mathfrak m_{R_1}}\tau\big(\|\Phi W(t_{m-1})\|_{W^{1,2}_x}^4+\|\Phi W(t_m)\|_{W^{1,2}_x}^4+\|{\bfu}_{m-1}\|_{W^{1,2}_x}^4+\|{\bfu}_m\|_{W^{1,2}_x}^4\big)\bigg]\leq\,cR_2(R_1)^4
\end{align*}
using also \eqref{lem:3.1b}.
We conclude by the discrete Gronwall lemma
 \begin{align*}
\E\bigg[\max_{1\leq n\leq \mathfrak m_{R_1}}\int_{\mt}|\nabla{\bfy}_{n}|^2\dx\bigg]&+\E\bigg[\sum_{m=1}^{\mathfrak m}\int_{\mt}|\nabla({\bfy}_{m}-{\bfy}_{m-1})|^2\dx\bigg]\\&+\frac{\mu}{2}\E\bigg[\sum_{m=1}^{\mathfrak m_{R_1}}\tau\int_{{\mt}}|\mathcal A{\bfy}_m|^2\dx\bigg]\\&\leq
\,ce^{cR_2(R_1)^4}\,\E\Bigl[\int_{\mt}|\nabla\bfu_{0}|^2\dx+1\Bigr]
\end{align*}
using that
\begin{align*}
\sum_{n=1}^{\mathfrak m_{R_1}}\tau_n\|{\bfy}_{n-1}\|_{L^2_x}^2\|\nabla{\bfy}_{n-1}\|_{L^2_x}^2%\leq\sum_{n=1}^m\tilde\tau_{n-1}^R\|\nabla{\bfy}_{n-1}\|_{L^2_x}^4
\leq R_2(R_1)^4
\end{align*}
by the definition of $\mathfrak m_{R_1}$.
This proves \eqref{i} for $q=1$.
% and proves the second estimate claimed in \eqref{iii}.
In order to obtain the estimate in \eqref{i} for $q=2$ we multiply \eqref{eq:2605} by $\int_{\mt}|\nabla{\bfy}_{m}|^2\dx$ and obtain
 \begin{align*}
\tfrac{1}{4}&\|\nabla\bfy_m\|_{L^2_x}^4+\tfrac{1}{4}\big(\|\nabla\bfy_m\|_{L^2_x}^2-\|\nabla\bfy_{m-1}\|_{L^2_x}^2\big)^2\\&+\tfrac{1}{2}\|\nabla(\bfy_m-\bfy_{m-1})\|_{L^2_x}^2\|\nabla\bfy_m\|_{L^2_x}^2 +\tfrac{\mu}{2}\tau\|\nabla\bfy_m\|_{L^2_x}^2\|\mathcal A\bfy_m\|_{L^2_x}^2\\
&\leq  
\tfrac{1}{4}\|\nabla\bfy_{m-1}\|_{L^2_x}^4+c\,\tau\|\nabla\bfy_m\|_{L^2_x}^2\big(\|\Phi W(t_{m-1})\|_{W^{1,2}_x}^4+\|\bfu_{m-1}\|_{W^{1,2}_x}^4\big)\\&+c\,\tau\|\nabla\bfy_m\|_{L^2_x}^2\big(\|\Phi W(t_m)\|_{W^{1,2}_x}^4+\|\bfu_m\|_{W^{1,2}_x}^4\big)+c\tau\|\bfy_{m-1}\|_{L^2_x}^2\|\nabla\bfy_{m-1}\|_{L^2_x}^2\|\nabla\bfy_m\|_{L^2_x}^4\\
&\leq  
\tfrac{1}{4}\|\nabla\bfy_{m-1}\|_{L^2_x}^4+c\,\tau\big(\|\Phi W(t_{m-1})\|_{W^{1,2}_x}^6+\|\bfu_{m-1}\|_{W^{1,2}_x}^6+\|\Phi W(t_m)\|_{W^{1,2}_x}^6+\|\bfu_m\|_{W^{1,2}_x}^6\big)\\&+c\tau\|\bfy_{m-1}\|_{L^2_x}^2\|\nabla\bfy_{m-1}\|_{L^2_x}^2\|\nabla\bfy_m\|_{L^2_x}^4.
\end{align*}
We can control again the second last term by means of  \eqref{lem:3.1b} and the last one with the help of the discrete Gronwall lemma (and the definition of ${\mathfrak m}_{R_1}$).
We obtain \eqref{i} for $q=2$. We can prove similarly the claim for $q\in\N$ by iteration following the strategy of  \cite[Lemma 3.1]{BCP}.\\ 
{\bf 2.} Ad \eqref{ii}. We apply $\nabla$ to (\ref{tdiscr-add}) and multiply eventually by $\nabla\mathcal A\bfy_m$ which yields
\begin{align}\label{part_2}
\begin{aligned}\tfrac{1}{2} \Bigl( \Vert {\mathcal A} \bfy_m\Vert^2_{L^2_x}& -  \Vert {\mathcal A} \bfy_{m-1}\Vert^2_{L^2_x} + \Vert {\mathcal A} \bigl(\bfy_m - \bfy_{m-1}\bigr)\Vert^2_{L^2_x}\Bigr) + \tfrac{{\tau}}{ \mu}{2} \Vert \nabla {\mathcal A}\bfy_m\Vert^2_{L^2_x}\\ 
&
\leq \,c{\tau}\Big( \Vert \nabla {\mathcal A}[\Phi W(t_m)]\Vert^2_{L^2_x}
%+\Vert \nabla\bigl[(\nabla \bfy_m) \bfy_m\bigr]\Vert^2_{L^2_x}+  \Vert \nabla\bigl[\mathcal L^{W(t_m)}(\bfy_m^R)\bigr]\Vert^2_{L^2_x}\Big)\\
+{\tau} \Big\langle\nabla\bigl[(\nabla \bfy_m) \bfy_m\bigr], \nabla {\mathcal A} \bfy_m \Big\rangle_{L^2_x}\\&+{\tau} \Big\langle\nabla\bigl[\mathcal P\mathcal L^{m}(\bfy_{m-1},\bfy_m)\bigr], \nabla {\mathcal A} {\bfy}_m \Big\rangle_{L^2_x}\Big).
\end{aligned}
\end{align}
%The right-hand side can be further bounded by
%\begin{align*}
%\leq \,c\tilde{\tau}_{m}^R\Big( \Vert \Phi_{\mathcal P}W(t_m)\Vert^2_{W^{3,2}_x}+ \Vert \Phi_{\mathcal P}W(t_m)\Vert_{W^{2,2}_x}^4+ \Vert \bfy_m^R\Vert_{W^{1,2}_x}^4
%+  \Vert \bfy_m^R\Vert^2_{W^{2,2}_x}\Vert \bfy_m^R\Vert^2_{L^\infty_x}\Big)
%\end{align*}
We deal independently with the terms on the right-hand side. First of all, we have for $\delta>0$
%we use $L^p$ interpolation
% repeatedly as for the corresponding term in \cite[p.~561]{BrDo} to conclude
%\footnote{Dominic, wuerde hier gerne Dein Argument auf S. 561, Mitte verwenden.}
\begin{align*} \Big\langle \nabla \bigl[(\nabla \bfy_m)\bfy_{m-1}\bigr],& \nabla {\mathcal A} \bfy_m \Big\rangle_{L^2_x}\\&\leq c(\delta)\|\nabla \bigl[(\nabla \bfy_m)\bfy_{m-1}\bigr]\|_{L^2_x}^2+\delta\|\nabla\A\bfy_m\|_{L^{2}_x}^2\\
&\leq c(\delta)\|\nabla^2 \bfy_m\|_{L^2_x}^2\|\bfy_{m-1}\|_{L^\infty_x}^2+\|\nabla \bfy_m\|_{L^4_x}^2\|\nabla \bfy_{m-1}\|_{L^4_x}^2+\delta\|\nabla\A\bfy_m\|_{L^{2}_x}^2\\
&\leq c(\delta)\|\bfy_m\|_{W^{2,2}_x}^2\|\bfy_{m-1}\|_{W^{2,2}_x}^2+\delta\|\nabla\A\bfy_m\|_{L^{2}_x}^2.
%&\leq c(\delta)\|{\bf u}_m\|_{W^{2,2}_x}^4+c(\delta)\|\Phi W(t_m)\|_{W^{2,2}_x}^4+\delta\|\nabla\A\bfy_m\|_{L^{2}_x}^2\\
%\leq {\color{blue} \frac{\tilde{\tau}_{m}^R \mu}{8} \Vert \nabla \pmb{\mathcal A} \bfy_m\Vert^2
%+ \frac{C\tilde{\tau}_{m}^R}{\mu} \Bigl(\Vert \bfy_m\Vert^{10}_{W^{1,2}_x} + \Vert \bfy_m\Vert^{10}_{L^{2}_x}\Bigr)}, 
\end{align*}
% By \cite[Lemma 2.1.20]{KukShi} (with $m=2$) and Young's inequality we have for all $i,j\in\{1,2\}$
%\begin{align*}
%\bigg|\Big\langle \partial_i\partial_j \bigl[(\nabla \bfy_m)\bfy_m\bigr], \partial_i\partial_j \bfy_m \Big\rangle_{L^2_x} \bigg|&\leq\,c\|\bfy_m\|_{W^{3,2}_x}^{\frac{7}{4}}\|\bfy_m\|_{W^{1,2}_x}^{\frac{3}{4}}\|\bfy_m\|_{L^2_x}^{\frac{1}{2}}\\
%&\leq\delta\|\bfy_m\|_{W^{3,2}_x}^2+c(\delta)\|\bfy_m\|_{W^{1,2}_x}^{6}\|\bfy_m\|_{L^2_x}^{4}\\
%&\leq\delta\|\bfy_m\|_{W^{3,2}_x}^2+c(\delta)\|\bfy_m\|_{W^{1,2}_x}^{10}.
%\end{align*}
%The first term can be controlled due to
%\begin{align*}
%\tilde{\tau}_{m}^R\E\Bigl[\sum_{m=1}^M\|\bfy_m\|_{W^{2,2}_x}^{4}\Bigr]\leq\,c\tilde{\tau}_{m}^R\E\Bigl[\sum_{m=1}^M\|{\bf u}_m\|_{W^{2,2}_x}^{4}\Bigr]+\,c\tilde{\tau}_{m}^R\E\Bigl[\sum_{m=1}^M\|\Phi W(t_m)\|_{W^{2,2}_x}^{4}\Bigr],
%\end{align*}
%which is bounded using \eqref{Lemma3.1BCP} and $\Phi\in L_2(\mathfrak U;W^{2,2}(\mt))$.
For the remaining term we use H\"older's inequality, Young's inequality and the interpolation $\Vert  {\mathcal A} \bfy_m\Vert^{3}_{L^2_x}\leq\Vert \nabla \bfy_m\Vert^{\frac{3}{2}}_{L^2_x} \Vert \nabla {\mathcal A} \bfy_m\Vert^{\frac{3}{2}}_{L^2_x}$
 to obtain
 \begin{align*} {\tau} \Big\langle\nabla &\bigl[(\nabla \bfy_m) [\Phi W(t_{m-1})]\bigr], \nabla{\mathcal A}\bfy_m\Big\rangle_{L^2_x} \\&\leq \frac{{\tau}\mu}{8} \Vert \nabla{\mathcal A} \bfy_m\Vert^2_{L^2_x} + c {\tau} \Bigl(\Vert \nabla \bfy_m \Vert^4_{L^2_x} + { \Vert \nabla [\Phi W(t_{m})]\Vert^4_{L^\infty_x}}\Bigr)  \\&\quad + c{\tau} \Bigl( \Vert{\mathcal A} \bfy_m\Vert^3_{L^2_x} + 
\Vert \Phi W(t_{m-1})\Vert^6_{L^\infty_x}\Bigr) \\
&\leq\frac{{\tau} \mu}{4} \Vert \nabla {\mathcal A} \bfy_m\Vert^2_{L^2_x} + c{\tau} \Bigl(1+\Vert \nabla \bfy_m \Vert^6_{L^2_x} + \Vert \Phi W(t_{m-1})\Vert^6_{W^{3,2}_x}\Bigr).\end{align*}
The last two terms can be controlled using \eqref{i} and $\Phi\in L_2(\mathfrak U;W^{3,2}(\mt;\R^2))$.
Similarly, we have
 \begin{align*} {\tau} \Big\langle\nabla &\bigl[(\nabla  [\Phi W(t_{m})])\bfy_{m-1}\bigr], \nabla{\mathcal A}\bfy_m\Big\rangle_{L^2_x} \\
&\leq\frac{{\tau} \mu}{4} \Vert \nabla {\mathcal A} \bfy_m\Vert^2_{L^2_x}%+\frac{{\tau} \mu}{4} \Vert \nabla {\mathcal A} \bfy_{m-1}\Vert^2_{L^2_x} 
+ c \tau \Bigl(\Vert \nabla \bfy_{m-1} \Vert^4_{W^{1,2}_x} + \Vert \Phi W(t_{m})\Vert^4_{W^{3,2}_x}\Bigr)\end{align*}
and
 \begin{align*} \tau \Big\langle\nabla &\bigl[(\nabla  [\Phi W(t_{m})]) [\Phi W(t_{m-1})])\bigr], \nabla{\mathcal A}\bfy_m\Big\rangle_{L^2_x} \\
&\leq\frac{{\tau}\mu}{4} \Vert \nabla {\mathcal A} \bfy_m\Vert^2_{L^2_x} + c {\tau}\Vert \Phi W(t_{m-1})\Vert^4_{W^{3,2}_x}+ c {\tau}\Vert \Phi W(t_{m})\Vert^4_{W^{3,2}_x},\end{align*}
which can be controlled accordingly.
Absorbing the $\Vert \nabla {\mathcal A} \bfy_m\Vert^2$-terms and iterating and applying the discrete Gronwall lemma we conclude
$${\mathbb E}\biggl[ \max_{1\leq m\leq \mathfrak m_{R_1}} \Vert \bfy_m\Vert^{2}_{W^{2,2}_x}\biggr] 
+ \mu  {\mathbb E} \biggl[ \sum_{m=1}^{\mathfrak m_{R_1}} {\tau}\Vert \nabla {\mathcal A} {\bf y}_m\Vert^2_{L^2_x}\biggr]\leq \,ce^{cR_2(R_1)^4}\,{\mathbb E}\biggl[ \Vert {\bf u}_0\Vert^{2}_{W^{3,2}_x}+1\biggr], $$
using that
\begin{align*}
\sum_{n=1}^{\mathfrak m_{R_1}}{\tau}\|\bfy_{n-1}\|_{W^{2,2}_x}^2%\leq\sum_{n=1}^m{\tau}\|\bfy_{n-1}\|_{W^{2,2}_x}^2
\leq R_2(R_1)^2
\end{align*}
by the definition of ${\mathfrak m}_{R_1}$.
This proves \eqref{ii} for $q=1$.
The proof for $q\geq2$ follows by multiplying \eqref{part_2} with $\Vert {\mathcal A} \bfy_m\Vert^{2^q-2}$ ($q \geq 1$) and proceeding recursively (using also \eqref{i}).

{\bf 3.} Ad \eqref{iii}. 
%The bound for the second sum in the claimed estimate was already obtained in the proof of (i). So, it remains to control the first sum. For this purpose
We apply $\nabla$ and $\|\cdot\|_{L^2_x}^2$ to \eqref{tdiscr-add} such that
\begin{align*}
\bigg\Vert \nabla\frac{{\bf y}_{m} - \bfy_{m-1}}{\tau}\bigg\Vert_{L^2_x}^2 &\leq\,c\,\big\Vert \nabla{\mathcal A} \bfy_m\big\Vert_{L^2_x}^2 +c\,\big\Vert\nabla{\mathcal A}[\Phi W(t_m)]\Vert_{L^2_x}^2 +c\,\big\Vert \nabla {\mathcal P}\bigl[(\nabla {\bf y}_{m}){\bf y}_{m-1} \bigr]\big\Vert_{L^2_x}^2\\&\quad +c\,\big\Vert\nabla{\mathcal P}\bigl[ \mathcal L^{m}(\bfy_{m-1},\bfy_m)\bigr]\big\Vert_{L^2_x}^2\\
&\leq\,c\,\big\Vert  \bfy_m\big\Vert_{W^{3,2}_x}^2 +c\,\big\Vert\Phi W(t_m)\Vert_{W^{3,2}_x}^2 +c\,\big\Vert  \bfy_{m-1}\big\Vert_{W^{2,2}_x}^4+c\,\big\Vert  \bfy_m\big\Vert_{W^{2,2}_x}^4\\&+c\,\big\Vert  \Phi W(t_{m-1})\big\Vert_{W^{2,2}_x}^4+c\,\big\Vert  \Phi W(t_m)\big\Vert_{W^{2,2}_x}^4,
\end{align*}
where we used Ladyshenskaya's inequality.
Multiplying by ${\tau}\|\bfy_m\|^{2^q-4}_{W^{2,2}_x}$ and summing with respect to $m$ yields 
{\rm (iii)} due to {\rm (ii)} and $\Phi\in L_2(\mathfrak U;W^{3,2}(\mt;\R^2))$. We can again iterate the argument for $q\geq3$.
\end{proof}
\begin{remark}
It is crucial that the diffusion coefficient in \eqref{eq:SNSy} is solenoidal in order to perform the integration by parts at the beginning of the proof of Lemma \ref{lem-rand1}. In the general case we must define $\bfy$ by $\bfy(t)=\bfu(t)-\mathcal P\Phi W(t)$ with the Helmholz-projection $\mathcal P$. Even if $\Phi$ vanishes
on $\partial\mathcal O$, we do not necessarily have the same for $\mathcal P\Phi W(t)$ (and $\bfy$).
\end{remark}

\subsection{Temporal Error analysis}
For the purpose of the error analysis
we define for $R>0$ and $m\in\{1,\dots,M\}$
\begin{align}\label{eq:tmR}
\tilde{\mathfrak t}_{m}^R:=t_m\wedge \tilde{\mathfrak t}_{R}^{\tt d},\quad \tilde\tau_{n}^{R}&=\tilde{\mathfrak t}_{n}^{R}-\tilde{\mathfrak t}_{n-1}^{R}.
\end{align} 
We consider the stopped process
\begin{align}\label{tdiscrtilde}
\tilde{\bfy}_m^{R}:=\begin{cases} \bfy_m,\quad \text{in}\quad \{t_m=\tilde{\mathfrak t}_m^{R}\},\\
\bfy_{{\mathfrak m}_{R}},\quad \text{in}\quad \{t_m> \tilde{\mathfrak{t}}_m^{R}\},
\end{cases}
\end{align}
where $\mathfrak m_{R}$ is defined before \eqref{eq:tRd}. The main effort of this section is to prove the following error estimate.
\begin{theorem}\label{thm:3.1tilde}
Assume that $\bfu_0\in L^r(\Omega,W^{3,2}(\mt;\R^2))\cap L^{5r}(\Omega,W^{1,2}_{0,\Div}(\mt;\R^2))$ for some $r\geq8$ and that $\Phi\in L_2(\mathfrak U;W^{1,2}_{0,\Div}\cap W^{3,2}(\mt;\R^2))$. Let $\bfy$ be the solution
to \eqref{eq:SNSy}, and $(\bfy_m)_{m=1}^M$ be the solution to \eqref{tdiscr-add}. Then we have the error estimate
\begin{align}\label{eq:thm:4}
\begin{aligned}
\max_{1\leq m\leq M}\E\bigg[\|\bfy(\tilde{\mathfrak{t}}^R_m)-\tilde\bfy^{R}_{m}\|^2_{L^2_x}&+\sum_{n=1}^m \tilde\tau_n^R \|\nabla\bfy(\tilde{\mathfrak{t}}^R_n)-\nabla\tilde\bfy^R_{n}\|^2_{L^2_x}\bigg]\leq \,c\,\tau^{2}e^{cR_2(R)^4}.
\end{aligned}
\end{align}
%for any $\varepsilon>0$.
\end{theorem}

The convergence in probability of the original scheme \eqref{tdiscr-add} is now a direct consequence of Theorem \ref{thm:3.1tilde}: Defining $R$ by $R_2(R)=c^{-1/4}\sqrt[4]{-\varepsilon \log \tau}$, where $\varepsilon>0$ is arbitrary, we have for any $\xi>0,\alpha<1$
\begin{align*}
&\max_{1\leq m\leq M}\mathbb P\bigg[\|\bfy(t_m)-\bfy_{m}\|_{L^2_x}^2+\sum_{n=1}^m \tau\|\nabla\bfy(t_n)-\nabla\bfy_{n}\|_{L^2_x}^2>\xi\,\tau^{2\alpha-2\varepsilon}\bigg]\\
&\leq \max_{1\leq m\leq M}\mathbb P\bigg[\|\bfy(\tilde{\mathfrak t}^R_{m})-\tilde\bfy_{m}^R\|_{L^2_x}^2+\sum_{n=1}^m \tilde\tau_{n}^R\|\nabla\bfy(\tilde{\mathfrak t}_{n}^R)-\nabla\tilde\bfy_{n}^R\|_{L^2_x}^2>\xi\,\tau^{2\alpha-2\varepsilon}\bigg]\\
&+\mathbb P\big[\{\tilde{\mathfrak t}_M^R<T\}\big]\rightarrow 0%\mathbb P\Big[\Big\{\omega\in\Omega:\tau\sum_{n=1}^\ell\|\bfu_n\|_{L^2_x}^{2q-2}\|\nabla\bfu_n\|_{L^2_x}^2\leq R^2\Big\}\Big]\rightarrow 0\\
%&+\mathbb P\Big[\Big\{\omega\in\Omega:\sum_{n=1}^\ell\tau_m^R\|\nabla\bfy_n\|_{L^2_x}^{2q+2}\leq R^2\Big\}\Big]\rightarrow 0
\end{align*}
as $\tau\rightarrow0$ by \eqref{eq:2008} (recall that $\mathfrak t_R\rightarrow \infty$ $\p$-a.s. by Theorem \ref{thm:inc2d}).
Relabeling $\alpha$ we have proved the following result.

\begin{theorem}\label{thm:maintdiscr}
Assume that $\bfu_0\in L^r(\Omega,W^{3,2}(\mt;\R^2))\cap L^{5r}(\Omega,W^{1,2}_{0,\Div}(\mt;\R^2))$ for some $r\geq8$ and that $\Phi\in L_2(\mathfrak U;W^{1,2}_{0,\Div}\cap W^{3,2}(\mt;\R^2))$. Let $$(\bfu,(\mathfrak{t}_R)_{R\in\N},\mathfrak{t})$$ be the unique global maximal strong solution to \eqref{eq:SNS} from Theorem \ref{thm:inc2dmax}.
Then we have for any $\xi>0$, $\alpha<1$ 
\begin{align*}
&\max_{1\leq m\leq M}\mathbb P\bigg[\|\bfy(t_m)-\bfy_{m}\|_{L^2_x}^2+\sum_{n=1}^m \tau\|\nabla\bfy(t_n)-\nabla\bfy_{n}\|_{L^2_x}^2>\xi\,\tau^{2\alpha}\bigg]\rightarrow 0
\end{align*}
as $\tau\rightarrow0$,
where $\bfy$ is the solution
to \eqref{eq:SNSy} and $(\bfy_{m})_{m=1}^M$ is the solution to \eqref{tdiscr-add}.
\end{theorem}

\begin{proof}[Proof of Theorem \ref{thm:3.1tilde}]
We integrate (\ref{eq:SNSy}) over $(\tilde{\mathfrak{t}}_{m-1}^R, \tilde{\mathfrak{t}}_m^R)$, and subtract this equation from (\ref{tdiscr-add}) (multiplied by $\tilde{\tau}_{m}^R$). Then ${\bf e}_\ell := {\bf y}(\tilde{\mathfrak t}^R_\ell) - \tilde{\bf y}^R_\ell$ (for $\ell \in \{ m-1,m\}$) solves
\begin{eqnarray}\nonumber\langle{\bf e}_m - {\bf e}_{m-1} - \mu \tilde{\tau}_{m}^R  {\mathcal A} {\bf e}_m,\bfphi \rangle_{L^2_x}
&=& 
\bigg\langle\mu \int_{\tilde{\mathfrak{t}}_{m-1}^R}^{\tilde{\mathfrak{t}}_{m}^R} \int_{s}^{\tilde{\mathfrak{t}}_{m}^R} \partial_t\nabla {\bf y}(\xi)\, {\rm d}\xi{\rm d}s ,\nabla\bfphi\bigg\rangle_{L^2_x} \\ 
\label{erro1}
&& + \bigg\langle\mu \int_{\tilde{\mathfrak{t}}_{m-1}^R}^{\tilde{\mathfrak{t}}_m^R}{\mathcal A} [\Phi (W_s - W_{\tilde{\mathfrak{t}}_m^R})]\, \dd s,\bfphi\bigg\rangle + \langle{\mathrm{NLT}}_m,\bfphi\rangle_{L^2_x}
\end{eqnarray}
for all $\bfphi\in W^{1,2}_{0,\Div}(\mt;\R^2)$,
where
\begin{align*}
{\mathrm{NLT}}_m &= \int_{\tilde{\mathfrak{t}}_{m-1}^R}^{\tilde{\mathfrak{t}}_m^R} \Bigl({\mathcal P} \Bigl[(\nabla {\bf y}){\bf y} + \mathcal L^{W}(\bfy)\Bigr] - {\mathcal P}\Bigl[ (\nabla \tilde{\bfy}_{m}^R)\tilde{\bfy}_{m-1}^R +
\mathcal L^{m}(\tilde{\bfy}_{m-1}^R,\tilde{\bfy}_{m}^R)\Bigr]\Bigr)\, {\rm d}s  \\
&=\int_{\tilde{\mathfrak{t}}_{m-1}^R}^{\tilde{\mathfrak{t}}_m^R} {\mathcal P} \Bigl[(\nabla {\bf y}){\bf y}- (\nabla \tilde{\bfy}_{m}^R)\tilde{\bfy}_{m-1}^R\Bigr]\, {\rm d}s +  \int_{\tilde{\mathfrak{t}}_{m-1}^R}^{\tilde{\mathfrak{t}}_m^R}{\mathcal P} \Bigl[\nabla\big[\Phi W_s]\bfy(s)-\nabla\big[\Phi W_{\tilde{\mathfrak{t}}_{m}^R}]\tilde{\bfy}_{m-1}^R\Bigr]\,\dd s\\
&+  \int_{\tilde{\mathfrak{t}}_{m-1}^R}^{\tilde{\mathfrak{t}}_m^R}{\mathcal P}\Bigl[(\nabla\bfy(s))\big[\Phi W_s\big]-\nabla\tilde{\bfy}_{m}^R\big[\Phi W_{\tilde{\mathfrak{t}}_{m-1}^R}]\Bigr]\, {\rm d}s\\&+  \int_{\tilde{\mathfrak{t}}_{m-1}^R}^{\tilde{\mathfrak{t}}_m^R}{\mathcal P}\Bigl[\nabla\big[\Phi W_s]\big[\Phi W_s]-\nabla\big[\Phi W_{\tilde{\mathfrak{t}}_{m}^R}]\big[\Phi W_{\tilde{\mathfrak{t}}_{m-1}^R}]\Bigr]\, {\rm d}s\\
&=:{\mathrm{NLT}}_m^1+{\mathrm{NLT}}_m^2+{\mathrm{NLT}}_m^3+{\mathrm{NLT}}_m^4.
\end{align*}
We test (\ref{erro1}) with ${\bf e}_m$ and apply expectations. The left-hand side then is
\begin{equation}\label{erro-1}\tfrac{1}{2} {\mathbb E}\bigl[ \big(\Vert {\bf e}_m\Vert^2_{L^2_x} - \Vert {\bf e}_{m-1}\Vert^2 _{L^2_x}+ \Vert {\bf e}_m - {\bf e}_{m-1}\Vert^2_{L^2_x}\big)\bigr] + \mu {\mathbb E}\bigl[ \tilde{\tau}_{m}^R\Vert \nabla {\bf e}_m\Vert^2_{L^2_x}\bigr] \, .
\end{equation}
The first term on the right-hand side of (\ref{erro1}) may be bounded by
\begin{eqnarray}\label{erro-2}
&&\delta{\mathbb E}\bigl[\tilde{\tau}_{m}^R  \Vert \nabla {\bf e}_m\Vert^2_{L^2_x} \bigr] +c(\delta) {\mathbb E}\Bigl[ \tilde{\tau}_{m}^R\Bigl(\int_{\tilde{\mathfrak{t}}_{m-1}^R}^{\tilde{\mathfrak{t}}_m^R} \Vert \nabla \partial_t {\bf y}(\xi)\Vert_{L^2_x}\, {\rm d}\xi\Bigr)^2\Bigr]\\ \nonumber
&&\qquad \leq \delta{\mathbb E}\bigl[\tilde{\tau}_{m}^R\Vert \nabla {\bf e}_m\Vert^2_{L^2_x} \bigr] +c(\delta){\tau}^2 {{\mathbb E}\biggl[ \int_{\tilde{\mathfrak{t}}_{m-1}^R}^{\tilde{\mathfrak{t}}_m^R} \Vert \nabla \partial_t {\bf y}\Vert^2_{L^2_x}\, {\rm d}\xi\biggr]},
\end{eqnarray}
where $\delta>0$ is arbitrary.
For the second term on the right-hand side of (\ref{erro1}), we use independence of increments of the  Wiener process to conclude that it equals 
\begin{eqnarray}\label{erro-3}
&&{\mathbb E}\Bigl[\mu \bigg\langle\int_{\tilde{\mathfrak{t}}_{m-1}^R}^{\tilde{\mathfrak{t}}_m^R}{\mathcal A} [\Phi (W_s - W_{\tilde{\mathfrak{t}}_m^R})]\, {\rm d}s, {\bf e}_m - {\bf e}_{m-1}\bigg\rangle_{L^2_x}\Bigr] \\ \nonumber
&&\qquad \leq \frac{1}{2} {\mathbb E}\bigl[\Vert {\bf e}_m - {\bf e}_{m-1}\Vert^2_{L^2_x}\bigr] + c\,{\mathbb E}\bigg[\bigg\| \int_{\tilde{\mathfrak{t}}_{m-1}^R}^{\tilde{\mathfrak{t}}_m^R} {\mathcal A} [\Phi (W_s - W_{\tilde{\mathfrak{t}}_m^R})]\, {\rm d}s\bigg\|_{L^2_x}^2\bigg],
\end{eqnarray}
where the last term is bounded by
$$ {\tau}^2 {\mathbb E}\bigg[\sup_{s \in [\tilde{\mathfrak{t}}_{m-1}^R,\tilde{\mathfrak{t}}_m^R]}\Vert \Phi(W_s - W_{\tilde{\mathfrak{t}}_m^R})\Vert^2_{W^{2,2}_x}\bigg]\leq {\tau}^2 {\mathbb E}\bigg[\sup_{s \in [t_{m-1},t_m]}\Vert \Phi(W_s - W_{t_m})\Vert^2_{W^{2,2}_x}\bigg]\leq c \tau^3$$
due to Doob's maximum inequality and $\Phi\in L_2(\mathfrak U;W^{2,2}(\mt;\R^2))$. Now we consider the errors due to nonlinear effects that are contained in ${\mathrm NLT}_m$: we start with the ``quadratic terms"
\begin{align*}
{\mathrm{NLT}}^1_m &= \int_{\tilde{\mathfrak{t}}_{m-1}^R}^{\tilde{\mathfrak{t}}_m^R} {\mathcal P}\Bigl[ \bigl(\nabla {\bf y}(s)\bigr) {\bf y}(s) -
\bigl(\nabla {\bf y}(\tilde{\mathfrak{t}}_m^R)\bigr) {\bf y}(\tilde{\mathfrak{t}}_m^R) + \bigl(\nabla {\bf y}(\tilde{\mathfrak{t}}_m^R)\bigr) {\bf y}(\tilde{\mathfrak{t}}_m^R) - \bigl( \nabla \tilde{\bf y}^R_m\bigr)\tilde{\bfy}_{m-1}^R\Bigr]\, {\rm d}s \\
&= \int_{\tilde{\mathfrak{t}}_{m-1}^R}^{\tilde{\mathfrak{t}}_m^R} {\mathcal P}\Bigl[ \bigl(\nabla [{\bf y}(s)-{\bf y}(\tilde{\mathfrak{t}}_m^R)]\bigr) {\bf y}(s) + \bigl(\nabla {\bf e}_m\bigr){\bf y}(\tilde{\mathfrak{t}}_m^R) + \bigl( \nabla \tilde{\bf y}^R_m\bigr){\bf e}_m\Bigr]\, {\rm d}s\\
&+\int_{\tilde{\mathfrak{t}}_{m-1}^R}^{\tilde{\mathfrak{t}}_m^R} {\mathcal P}\Bigl[\bigl( \nabla \tilde{\bf y}^R_m\bigr)(\tilde{\bf y}^R_m-\tilde\bfy^R_{m-1})+\nabla \bfy(\tilde{\mathfrak{t}}_m^R)(\bfy(s)-\bfy(\tilde{\mathfrak{t}}_m^R))\Bigr]\, {\rm d}s\\
& =: \sum_{i=1}^5 {\mathrm{NLT}}^{1,i}_m\, .
\end{align*}
Note that $\langle{\mathrm{NLT}}^{1,2}_m,\bfe_m\rangle_{L^2_x}=0$ such that we only have to consider
${\mathrm{NLT}}^{1,1}_m$, ${\mathrm{NLT}}^{1,3}_m$, ${\mathrm{NLT}}^{1,4}_m$ and ${\mathrm{NLT}}^{1,5}_m$.
Thanks to ${\bf y}(s) - {\bf y}(\tilde{\mathfrak{t}}_m^R) = \int_{\tilde{\mathfrak{t}}_m^R}^s \partial_t {\bf y}(\xi)\, {\rm d}\xi$, we re-write $\langle{\mathrm{NLT}}^{1,1}_m, {\bf e}_m\rangle_{L^2_x}$ as
\begin{align*}
\bigg\langle &\int_{\tilde{\mathfrak{t}}_{m-1}^R}^{\tilde{\mathfrak{t}}_m^R} \Bigl(\int_{\tilde{\mathfrak{t}}_m^R}^s \nabla \partial_t {\bf y}(\xi)\, {\rm d}\xi\Bigr) {\bf y}(s)\, {\rm d}s, {\bf e}_m  \bigg\rangle_{L^2_x}\\& \leq \tilde{\tau}_{m}^R \sup_{\tilde{\mathfrak{t}}_{m-1}^R \leq s \leq \tilde{\mathfrak{t}}_m^R}\biggl(\bigg\Vert\int_s^{\tilde{\mathfrak{t}}_m^R} \nabla \partial_t {\bf y}(\xi)\, {\rm d}\xi\bigg\Vert_{L^2_x} \Vert {\bf y}(s)\Vert_{L^4_x}\biggr) \Vert {\bf e}_m\Vert_{L^4_x} \\
& \leq \delta\tilde{\tau}_{m}^R \Vert \nabla {\bf e}_m\Vert^2_{L^2_x} +  c(\delta) {\tau}^2  \int_{\tilde{\mathfrak{t}}_{m-1}^R}^{\tilde{\mathfrak{t}}_m^R} \Vert \nabla \partial_t {\bf y}(s)\Vert^2_{L^2_x}\, {\rm d}s \sup_{s \in [\tilde{\mathfrak{t}}_{m-1}^R,\tilde{\mathfrak{t}}_m^R]}\Vert \nabla{\bf y}(s)\Vert^2_{L^2_x}\\
& \leq \delta\tilde{\tau}_{m}^R \Vert \nabla {\bf e}_m\Vert^2_{L^2_x} +  c(\delta) {\tau}^2 R_2(R)^2 \int_{\tilde{\mathfrak{t}}_{m-1}^R}^{\tilde{\mathfrak{t}}_m^R} \Vert \nabla \partial_t {\bf y}(s)\Vert^2_{L^2_x}\, {\rm d}s
\end{align*}
for $\delta>0$
using Corollary \ref{cor:add} and recalling that $\tilde{\mathfrak t}_m^R\leq\mathfrak t_R$.
Moreover, we have
\begin{align*}
\langle{\mathrm{NLT}}^{1,3}_m, {\bf e}_m\rangle_{L^2_x}& \leq \tilde{\tau}_{m}^R \Vert {\bf e}_m\Vert^2_{L^4_x}  \Vert \nabla\tilde{\bf y}^R_m\Vert_{L^2_x}\leq \,c\tilde{\tau}_{m}^R \Vert {\bf e}_m\Vert_{L^2_x}  \Vert \nabla{\bf e}_m\Vert_{L^2_x}  \Vert \nabla\tilde{\bf y}^R_m\Vert_{L^2_x}\\
& \leq \delta\tilde{\tau}_{m}^R \Vert \nabla {\bf e}_m\Vert^2_{L^2_x} + 
c(\delta) \tilde{\tau}_{m}^R R_2(R)^2 \|\bfe_m\|_{L^2_x}^2
\end{align*}
by definition of $\tilde{\mathfrak t}_m^R$.
Arguing similarly as for the estimates for ${\mathrm{NLT}}^{1,1}_m$ above and using Sobolev's embedding $W_0^{1,2}(\mt)\hookrightarrow L^4(\mt)$ we have
\begin{align*}
\langle{\mathrm{NLT}}^{1,5}_m, {\bf e}_m\rangle_{L^2_x}&=\bigg\langle \int_{\tilde{\mathfrak{t}}_{m-1}^R}^{\tilde{\mathfrak{t}}_m^R} \nabla{\bf y}(\tilde{\mathfrak t}_m^R)\Bigl(\int_{\tilde{\mathfrak{t}}_m^R}^s \partial_t {\bf y}(\xi)\, {\rm d}\xi\Bigr) \, {\rm d}s, {\bf e}_m  \bigg\rangle_{L^2_x}\\& \leq \tilde{\tau}_{m}^R \sup_{\tilde{\mathfrak{t}}_{m-1}^R \leq s \leq \tilde{\mathfrak{t}}_m^R}\biggl(\bigg\Vert\int_s^{\tilde{\mathfrak{t}}_m^R}  \partial_t {\bf y}(\xi)\, {\rm d}\xi\bigg\Vert_{L^4_x} \Vert \nabla{\bf y}(s)\Vert_{L^2_x}\biggr) \Vert {\bf e}_m\Vert_{L^4_x} \\
& \leq \delta\tilde{\tau}_{m}^R \Vert \nabla {\bf e}_m\Vert^2_{L^2_x} +  c(\delta) {\tau}^2  \int_{\tilde{\mathfrak{t}}_{m-1}^R}^{\tilde{\mathfrak{t}}_m^R} \Vert \nabla \partial_t {\bf y}(s)\Vert^2_{L^2_x}\, {\rm d}s \sup_{s \in [\tilde{\mathfrak{t}}_{m-1}^R,\tilde{\mathfrak{t}}_m^R]}\Vert \nabla{\bf y}(s)\Vert^2_{L^2_x}
\end{align*}
such that we obtain the same estimate for its expectation.
For ${\mathrm{NLT}}^{1,4}_m$ we employ a discrete version of this argument to obtain
\begin{align*}
\langle{\mathrm{NLT}}^{1,4}_m, {\bf e}_m\rangle_{L^2_x}& \leq \tilde{\tau}_{m}^R \big\Vert\tilde{\bf y}^R_m-\tilde\bfy^R_{m-1}\big\Vert_{L^4_x} \Vert \nabla\tilde{\bf y}^R_m\Vert_{L^2_x} \Vert {\bf e}_m\Vert_{L^4_x} \\
& \leq \delta\tilde{\tau}_{m}^R \Vert \nabla {\bf e}_m\Vert^2_{L^2_x} +  c(\delta) {\tau}^3   \Big\Vert \frac{\nabla(\tilde{\bf y}^R_m-\tilde\bfy^R_{m-1})}{\tau}\Big\Vert^2_{L^2_x}\Vert \nabla\tilde{\bf y}^R_m\Vert^2_{L^2_x}.
\end{align*}
Note that the last term can be controlled by means of
Lemma \ref{lem-rand1} (iii) giving the bound $c(\delta)\tau^2e^{cR_2(R)^4}$ for its expectation (in summed form).

For the second term in $\mathrm{NLT}$ we write
\begin{align*}
\langle {\mathrm{NLT}}^{2}_m,\bfe_{m}\rangle_{L^2_x}&=\bigg\langle\int_{\tilde{\mathfrak{t}}_{m-1}^R}^{\tilde{\mathfrak{t}}_m^R}\big(\nabla\big[\Phi W_s]\bfy(s)-\nabla\big[\Phi W_{\tilde{\mathfrak{t}}_{m}^R}]\tilde{\bfy}_{m-1}^R\big)\,\dd s,\bfe_m\bigg\rangle_{L^2_x}\\
&=\bigg\langle\int_{\tilde{\mathfrak{t}}_{m-1}^R}^{\tilde{\mathfrak{t}}_m^R}\nabla\big[\Phi (W_s-W_{\tilde{\mathfrak{t}}_{m}^R})]\bfy(s)\,\dd s,\bfe_m\bigg\rangle_{L^2_x}\\&+\bigg\langle\int_{\tilde{\mathfrak{t}}_{m-1}^R}^{\tilde{\mathfrak{t}}_m^R}\nabla\big[\Phi W_{\tilde{\mathfrak{t}}_{m-1}^R}]\bfe_{m-1}\,\dd s,\bfe_m\bigg\rangle_{L^2_x}\\
&+\bigg\langle\int_{\tilde{\mathfrak{t}}_{m-1}^R}^{\tilde{\mathfrak{t}}_m^R}\nabla\big[\Phi (W_{\tilde{\mathfrak{t}}_m^R}-W_{\tilde{\mathfrak{t}}_{m-1}^R})]\bfe_{m-1}\,\dd s,\bfe_m\bigg\rangle_{L^2_x}\\&+\bigg\langle\int_{\tilde{\mathfrak{t}}_{m-1}^R}^{\tilde{\mathfrak{t}}_m^R}\nabla\big[\Phi W_{\tilde{\mathfrak{t}}_m^R}](\bfy(s)-\bfy(\tilde{\mathfrak{t}}_{m-1}^R)\big)\,\dd s,\bfe_m\bigg\rangle_{L^2_x}\\
&=:{\mathrm{E}}^{2,1}_m+{\mathrm{E}}^{2,2}_m+{\mathrm{E}}^{2,3}_m+{\mathrm{E}}^{2,4}_m.
\end{align*}
For the first term we split further
\begin{align*}
%\bigg\langle&\int_{\tilde{\mathfrak{t}}_{m-1}^R}^{\tilde{\mathfrak{t}}_m^R}\nabla\big[\Phi (W_s-W_{\tilde{\mathfrak{t}}_m^R})]\bfy(s)\,\dd s,\bfe_m\bigg\rangle_{L^2_x}\\
{\mathrm{E}}^{2,1}_m&=\bigg\langle\int_{\tilde{\mathfrak{t}}_{m-1}^R}^{\tilde{\mathfrak{t}}_m^R}\nabla\big[\Phi (W_s-W_{\tilde{\mathfrak{t}}_m^R})]\bfy(s)\,\dd s,\bfe_m-\bfe_{m-1}\bigg\rangle_{L^2_x}\\
&+\bigg\langle\int_{\tilde{\mathfrak{t}}_{m-1}^R}^{\tilde{\mathfrak{t}}_m^R}\nabla\big[\Phi (W_s-W_{\tilde{\mathfrak{t}}_m^R})]\big(\bfy(s)-\bfy(\tilde{\mathfrak t}^R_{m-1})\big)\,\dd s,\bfe_{m-1}\bigg\rangle_{L^2_x}\\
&+\bigg\langle\int_{\tilde{\mathfrak{t}}_{m-1}^R}^{\tilde{\mathfrak{t}}_m^R}\nabla\big[\Phi (W_s-W_{\tilde{\mathfrak{t}}_m^R})]\bfy(\tilde{\mathfrak{t}}_{m-1}^R)\,\dd s,\bfe_{m-1}\bigg\rangle_{L^2_x}\\
&=:{\mathrm{E}}^{2,1,1}_m+{\mathrm{E}}^{2,1,2}_m+{\mathrm{E}}^{2,1,3}_m.
\end{align*}
We notice that $\E\big[\sum_{n=1}^m{\mathrm{E}}^{2,1,3}_n\big]=0$.
Indeed, $\sum_{n=1}^m{\mathrm{E}}^{2,1,3}_n$ can be written as the decomposition of the $(\mathfrak F_{t_m})$-martingale
\begin{align*}
\sum_{n=1}^m\bigg\langle\int_{t_{n-1}}^{t_n}\nabla\big[\Phi (W_s-W_{t_n})]\bfy(t_{m-1})\,\dd s,\bfe_{m-1}\bigg\rangle_{L^2_x}
\end{align*}
with the $(\mathfrak F_{t_m})$-stopping time $\mathfrak m_R$ defined after \eqref{tdiscrtilde}. Hence it is an $(\mathfrak F_{t_m})$-martingale in its own right.
We estimate
\begin{align*}
\E\big[&{\mathrm{E}}^{2,1,1}_m\big]\\&\leq\,\E\bigg[\tilde{\tau}_{m}^R\sup_{s\in[\tilde{\mathfrak{t}}_{m-1}^R,\tilde{\mathfrak{t}}_m^R]}\|\nabla\big[\Phi (W_s-W_{\tilde{\mathfrak{t}}_m^R})]\|_{L^4_x}\sup_{s\in[\tilde{\mathfrak{t}}_{m-1}^R,\tilde{\mathfrak{t}}_m^R]}\|\bfy(s)\|_{L^4_x}\|\bfe_m-\bfe_{m-1}\|_{L^2_x}\bigg]\\
&\leq\,c(\delta){\tau}^2\E\bigg[\sup_{s\in[\tilde{\mathfrak{t}}_{m-1}^R,\tilde{\mathfrak{t}}_m^R]}\|\nabla^2\big[\Phi (W_s-W_{\tilde{\mathfrak{t}}_m^R})]\|^2_{L^2_x}\sup_{s\in[\tilde{\mathfrak{t}}_{m-1}^R,\tilde{\mathfrak{t}}_m^R]}\|\nabla\bfy(s)\|_{L^2_x}^2\bigg]\\&+\delta\E\big[\|\bfe_m-\bfe_{m-1}\|^2_{L^2_x}\big]\\
&\leq\,c(\delta){\tau}^2R_2(R)^2\E\bigg[\sup_{s\in[\tilde{\mathfrak{t}}_{m-1}^R,\tilde{\mathfrak{t}}_m^R]}\|\nabla^2\big[\Phi (W_s-W_{\tilde{\mathfrak{t}}_m^R})]\|^2_{L^2_x}\bigg]\\&+\delta\E\big[\|\bfe_m-\bfe_{m-1}\|^2_{L^2_x}\big]\\
&\leq\,c(\delta){\tau}^3R_2(R)^2+\delta\E\big[\|\bfe_m-\bfe_{m-1}\|^2_{L^2_x}\big],
\end{align*}
using Doob's maximal inequality, $\Phi\in L_2(\mathfrak U;W^{2,2}(\mt;\R^2))$ and Corollary \ref{cor:add} (recalling that $\tilde{\mathfrak t}_m^R\leq\mathfrak t_R$). Similarly, it holds
\begin{align*}
\E\big[&{\mathrm{E}}^{2,1,2}_m\big]\\
&\leq\,c(\delta)\tilde{\tau}_{m}^R\E\bigg[\sup_{s\in[\tilde{\mathfrak{t}}_{m-1}^R,\tilde{\mathfrak{t}}_m^R]}\|\nabla\big[\Phi (W_s-W_{\tilde{\mathfrak{t}}_m^R})]\|^2_{L^4_x}\sup_{s\in[\tilde{\mathfrak{t}}_{m-1}^R,\tilde{\mathfrak{t}}_m^R]}\|\bfy(s)-\bfy({\tilde{\mathfrak{t}}_{m-1}^R})\|_{L^2_x}^2\bigg]\\&+\delta\tilde{\tau}_{m}^R\E\big[\|\bfe_{m-1}\|^2_{L^4_x}\big]\\
&\leq\,c(\delta)\tau^2\bigg(\E\bigg[\sup_{s\in[\tilde{\mathfrak{t}}_{m-1}^R,\tilde{\mathfrak{t}}_m^R]}\|\nabla^2\big[\Phi (W_s-W_{\tilde{\mathfrak{t}}_m^R})]\|^4_{L^2_x}\bigg]\bigg)^{\frac{1}{2}}\bigg(\E\bigg[\int_{\tilde{\mathfrak{t}}_{m-1}^R}^{\tilde{\mathfrak{t}}_m^R}\|\partial_t\bfy(\xi)\|_{L^2_x}^4\,\dd\xi\bigg]\bigg)^{\frac{1}{2}}\\&+\delta\tilde{\tau}_{m}^R\E\big[\|\nabla\bfe_{m-1}\|^2_{L^2_x}\big]\\
&\leq\,c(\delta){\tau}^3R^4+\delta\tilde{\tau}_{m}^R\E\big[\|\nabla\bfe_{m-1}\|^2_{L^2_x}\big],
\end{align*}
using Corollary \ref{cor:add}.
Furthermore, we obtain
\begin{align*}
\E\big[{\mathrm{E}}^{2,2}_m\big]
%&\leq\,\tilde{\tau}_{m}^R\bigg(\E\bigg[\|\nabla\big[\Phi (W_{\tilde{\mathfrak{t}}_m^R})]\bfe_m\|_{L^2_x}^2\bigg]\bigg)^{\frac{1}{2}}\bigg(\E\|\bfe_m\|_{L^2_x}^2\bigg)^{\frac{1}{2}}\\
&\leq\,\tilde{\tau}_{m}^R\E\bigg[\|\nabla\Phi (W_{\tilde{\mathfrak{t}}_{m-1}^R})\|_{L^4_x}\|\bfe_{m-1}\|_{L^2_x}\|\bfe_m\|_{L^4_x}\bigg]\\
%&\leq\,\tilde{\tau}_{m}^R\bigg(\E\bigg[\Omega_{m-1}^{\tilde{\tau}_{m}^R,\varepsilon}\|\Phi (W_{\tilde{\mathfrak{t}}_{m-1}^R})\|^4_{L_2(\mathfrak U;W^{1,2}(\mt))}\bigg]\bigg)^{\frac{1}{4}}\bigg(\E\|\bfe_m\|_{W^{1,2}_x}^4\bigg)^{\frac{1}{4}}\big(\E\big[\Omega_{m-1}^{\tilde{\tau}_{m}^R,\varepsilon}\|\bfe_m\|_{L^2_x}^2\big]\big)^{\frac{1}{2}}\\
&\leq\,c\tilde{\tau}_{m}^RR_2(R)\E\big[\|\bfe_m\|_{L^2_x}\|\bfe_m\|_{L^2_x}^{\frac{1}{2}}\|\nabla\bfe_m\|_{L^2_x}^{\frac{1}{2}}\big]\\&\leq\,c(\delta)\tilde{\tau}_{m}^RR_2(R)^2\E\big[\|\bfe_{m-1}\|_{L^2_x}^2+\|\bfe_m\|_{L^2_x}^2\big]+\delta\tilde{\tau}_{m}^R\E\big[\|\nabla\bfe_m\|_{L^2_x}^2\big]
\end{align*}
by definition of $\tilde{\mathfrak t}_m^R$.
The term ${\mathrm{E}}^{2,3}_m$ has to be splitted in the same fashion as
\begin{align*}
{\mathrm{E}}^{2,3}_m&=\bigg\langle\int_{\tilde{\mathfrak{t}}_{m-1}^R}^{\tilde{\mathfrak{t}}_m^R}\nabla\big[\Phi (W_{\tilde{\mathfrak{t}}_m^R}-W_{\tilde{\mathfrak{t}}_{m-1}^R})]\bfe_m\,\dd s,\bfe_m-\bfe_{m-1}\bigg\rangle_{L^2_x}\\
%&+\bigg\langle\int_{\tilde{\mathfrak{t}}_{m-1}^R}^{\tilde{\mathfrak{t}}_m^R}\nabla\big[\Phi (W_{\tilde{\mathfrak{t}}_m^R}-W_{\tilde{\mathfrak{t}}_{m-1}^R})](\bfe_m-\bfe_{m-1})\,\dd s,\bfe_{m-1}\bigg\rangle_{L^2_x}\\
&+\bigg\langle\int_{\tilde{\mathfrak{t}}_{m-1}^R}^{\tilde{\mathfrak{t}}_m^R}\nabla\big[\Phi (W_{\tilde{\mathfrak{t}}_m^R}-W_{\tilde{\mathfrak{t}}_{m-1}^R})]\bfe_{m-1}\,\dd s,\bfe_{m-1}\bigg\rangle_{L^2_x}\\
&=:{\mathrm{E}}^{2,3,1}_m+{\mathrm{E}}^{2,3,2}_m,
\end{align*}
where again $\E\big[\sum_{n=1}^m{\mathrm{E}}^{2,3,2}_n\big]=0$.\footnote{Similarly to ${\mathrm{E}}^{2,1,3}_n$ we can write $\sum_{n=1}^m{\mathrm{E}}^{2,3,2}_n$ as the decomposition of a discrete martingale with $\mathfrak m_R$.}
Proceeding similarly to the estimates of ${\mathrm{E}}^{2,1}_m$ we have
\begin{align*}
\E\big[{\mathrm{E}}^{2,3,1}_m\big]
&\leq\,c(\delta){\tau}^2\bigg(\E\bigg[\sup_{s\in[\tilde{\mathfrak{t}}_{m-1}^R,\tilde{\mathfrak{t}}_m^R]}\|\nabla^2\big[\Phi (W_s-W_{\tilde{\mathfrak{t}}_m^R})]\|^4_{L^2_x}\bigg]\bigg)^{\frac{1}{2}}\big(\E\|\nabla\bfe_m\|_{L^2_x}^4\big)^{\frac{1}{2}}\\&+\delta\big[\E\|\bfe_m-\bfe_{m-1}\|^2_{L^2_x}\big]\\
&\leq\,c(\delta){\tau}^3R_2(R)^2+\delta\E\big[\|\bfe_m-\bfe_{m-1}\|^2_{L^2_x}\big].
\end{align*}
%and the same estimate holds for $\mathrm{E}^{2,3,2}_m$.
%using
%\begin{align*}
%{\mathrm{E}}^{2,3,2}_m=\bigg\langle\bfe_m-\bfe_{m-1},\int_{\tilde{\mathfrak{t}}_{m-1}^R}^{\tilde{\mathfrak{t}}_m^R}\big(\nabla\big[\Phi (W_{\tilde{\mathfrak{t}}_m^R}-W_{\tilde{\mathfrak{t}}_{m-1}^R})]\big)^*\bfe_{m-1}\,\dd s\bigg\rangle_{L^2_x}.
%\end{align*}
Finally, we have
\begin{align*}
\E&\big[{\mathrm{E}}^{2,4}_m\big]\\
&\leq\,\tilde{\tau}_{m}^R\bigg(\E\bigg[\|\nabla\big[\Phi (W_{\tilde{\mathfrak{t}}_m^R})]\|^4_{L^4_x}\bigg]\bigg)^{\frac{1}{4}}\bigg(\E\|\bfy(s)-\bfy(\tilde{\mathfrak{t}}_{m-1}^R)\|_{L^{2}_x}^4\bigg)^{\frac{1}{4}}\big(\E\big[\|\nabla\bfe_m\|_{L^2_x}^2\big]\big)^{\frac{1}{2}}\\&\leq\,c{\tau}^{2}\bigg(\E\int_{\tilde{\mathfrak{t}}_{m-1}^R}^s\|\partial_t\bfy(\xi)\|_{L^{2}_x}^4\,\dd\xi\bigg)^{\frac{1}{2}}\big(\E\big[\|\nabla\bfe_m\|_{L^2_x}^2\big]\big)^{\frac{1}{2}}\leq\,c(\delta){\tau}^4R^4+\delta\E\big[\tilde{\tau}_{m}^R\|\nabla\bfe_m\|_{L^2_x}^2\big]
\end{align*}
using Corollary \ref{cor:add} and $\Phi\in L_2(\mathfrak U;W^{2,2}(\mt;\R^2))$.
The estimates for $\mathrm{NLT}^{3}_m$ are analogues and lead to the same result.
%Due to 
%\begin{align*}
%\bigg\langle&\int_{\tilde{\mathfrak{t}}_{m-1}^R}^{\tilde{\mathfrak{t}}_m^R}\mathcal P\big[\mathcal L^{W}(\bfy)-\mathcal L^{W(\tilde{\mathfrak{t}}_m^R)}(\bfy)\big]\,\dd s,\bfe_m\bigg\rangle\\& \leq \frac{\tilde{\tau}_{m}^R}{2} \Vert \nabla {\bf e}_m\Vert^2_{L^2_x} + 
%C \int_{\tilde{\mathfrak{t}}_{m-1}^R}^{\tilde{\mathfrak{t}}_m^R}\Vert \mathcal L^{W(s)}(\bfy(s))-\mathcal L^{W(\tilde{\mathfrak{t}}_m^R)}(\tilde{\bfy}_{m}^R)\Vert^2_{L^2_x}\,\dd s
%\end{align*}
For the final term we split
\begin{align*}
\langle{\mathrm{NLT}}^{4}_m,\bfe_{m}\rangle_{L^2_x}&=\bigg\langle\int_{\tilde{\mathfrak{t}}_{m-1}^R}^{\tilde{\mathfrak{t}}_m^R}\big(\nabla\big[\Phi W_s]\Phi W_s-\nabla\big[\Phi W_{\tilde{\mathfrak{t}}_m^R}]\Phi W_{\tilde{\mathfrak{t}}_{m-1}^R}\big)\,\dd s,\bfe_m-\bfe_{m-1}\bigg\rangle_{L^2_x}\\
&+\bigg\langle\int_{\tilde{\mathfrak{t}}_{m-1}^R}^{\tilde{\mathfrak{t}}_m^R}\big(\nabla\big[\Phi W_s]\Phi W_s-\nabla\big[\Phi W_{\tilde{\mathfrak{t}}_m^R}]\Phi W_{\tilde{\mathfrak{t}}_{m-1}^R}\big)\,\dd s,\bfe_{m-1}\bigg\rangle_{L^2_x}\\
&=:{\mathrm{E}}^{4,1}_m+{\mathrm{E}}^{4,2}_m.
\end{align*}
By Young's inequality we have for $\delta>0$ arbitrary
\begin{align*}
%\E\bigg\langle&\int_{\tilde{\mathfrak{t}}_{m-1}^R}^{\tilde{\mathfrak{t}}_m^R}\mathcal P\big[\mathcal L^{W}_3-\mathcal L^{W(\tilde{\mathfrak{t}}_m^R)}_3\big]\,\dd s,\bfe_m\bigg\rangle_{L^2_x}\\&=\E\bigg\langle\int_{\tilde{\mathfrak{t}}_{m-1}^R}^{\tilde{\mathfrak{t}}_m^R}\big(\nabla\big[\Phi W_s]\Phi W_s-\nabla\big[\Phi W_{\tilde{\mathfrak{t}}_m^R}]\Phi W_{\tilde{\mathfrak{t}}_m^R}\big)\,\dd s,\bfe_m\bigg\rangle_{L^2_x}\\
%&=\E\bigg\langle\int_{\tilde{\mathfrak{t}}_{m-1}^R}^{\tilde{\mathfrak{t}}_m^R}\big(\nabla\big[\Phi W_s]\Phi W_s-\nabla\big[\Phi W_{\tilde{\mathfrak{t}}_m^R}]\Phi W_{\tilde{\mathfrak{t}}_m^R}\big)\,\dd s,\bfe_m-\bfe_{m-1}\bigg\rangle_{L^2_x}\\
\E&[{\mathrm{E}}^{4,1}_m]\\&\leq \delta\E[\|\bfe_m-\bfe_{m-1}\|^2_{L^2_x}]+c(\delta)\E\bigg\|\int_{\tilde{\mathfrak{t}}_{m-1}^R}^{\tilde{\mathfrak{t}}_m^R}\big(\nabla\big[\Phi W_s]\Phi W_s-\nabla\big[\Phi W_{\tilde{\mathfrak{t}}_m^R}]\Phi W_{\tilde{\mathfrak{t}}_{m-1}^R}\big)\,\dd s\bigg\|^2_{L^2_x}\\
&\leq \delta\E[\|\bfe_m-\bfe_{m-1}\|^2_{L^2_x}]+c(\delta){\tau}^2\E\bigg[\sup_{s\in[\tilde{\mathfrak{t}}_{m-1}^R,\tilde{\mathfrak{t}}_m^R]}\big\|\nabla\big[\Phi W_s]\Phi W_s-\nabla\big[\Phi W_{\tilde{\mathfrak{t}}_m^R}]\Phi W_{\tilde{\mathfrak{t}}_{m-1}^R}\big\|^2_{L^2_x}\bigg]\\
&\leq \delta\E[\|\bfe_m-\bfe_{m-1}\|^2_{L^2_x}]+c(\delta){\tau}^{3}R_2(R)^2,
\end{align*}
using also Doob's maximal inequality, the definition of $\tilde{\mathfrak{t}}_m^R$ and $\Phi\in L_2(\mathfrak U;W^{3,2}(\mt;\R^2))$ (together with the Sobolev embedding $W^{3,2}(\mt)\hookrightarrow W^{1,\infty}(\mt)$). In order to estimate
${\mathrm{E}}^{4,2}$
we write
\begin{align*}
\E&[{\mathrm{E}}^{4,2}]\\&=\E\bigg[\bigg\langle\E\bigg[\int_{\tilde{\mathfrak{t}}_{m-1}^R}^{\tilde{\mathfrak{t}}_m^R}\nabla\big[\Phi W_s]\Phi W_s-\nabla\big[\Phi W_{\tilde{\mathfrak{t}}_m^R}]\Phi W_{\tilde{\mathfrak{t}}_{m-1}^R}\Big|\mathfrak F_{{t}_{m-1}}\bigg]\,\dd s,\bfe_{m-1}\bigg\rangle_{L^2_x}\bigg]\\
&\leq\E\bigg[\bigg\langle\E\bigg[\int_{\tilde{\mathfrak{t}}_{m-1}^R}^{\tilde{\mathfrak{t}}_m^R}\big\|\nabla\big[\Phi W_s]\Phi W_s-\nabla\big[\Phi W_{\tilde{\mathfrak{t}}_m^R}]\Phi W_{\tilde{\mathfrak{t}}_{m-1}^R}\big\|_{L^2_x}\Big|\mathfrak F_{{t}_{m-1}}\bigg]\,\dd s\,\|\bfe_{m-1}\|_{L^2_x}\bigg]\\
&\leq\delta\E\big[\tilde{\tau}_{m}^R\|\bfe_{m-1}\|_{L^2_x}^2\big]+c(\delta)\tau^2\E\bigg[\E\bigg[\sup_{t,s,\tau\in[t_{m-1},t_m]}\big\|\nabla\big[\Phi W_s]\Phi W_s-\nabla\big[\Phi W_{t}]\Phi W_{\tau}\big\|_{L^2_x}^2\bigg]\Big|\mathfrak F_{t_{m-1}}\bigg]\\
&=\delta\E\big[\|\bfe_{m-1}\|_{L^2_x}^2\big]+c(\delta)\tau^2\E\bigg[\sup_{t,s,\tau\in[t_{m-1},t_m]}\big\|\nabla\big[\Phi W_s]\Phi W_s-\nabla\big[\Phi W_{t}]\Phi W_{\tau}\big\|_{L^2_x}^2\bigg].
\end{align*}
Using Doob's maximal inequality and the assumption $\Phi\in L_2(\mathfrak U;W^{3,2}(\mt;\R^2))$ (as in the estimates for ${\mathrm{E}}^{4,1}_m$)
%Since $W$ is a Markov process 
%we have for $s\in[\tilde{\mathfrak{t}}_{m-1}^R,\tilde{\mathfrak{t}}_m^R]$
%\begin{align*}
%\E\Big[&\nabla\big[\Phi W_s]\Phi W_s\Big|\mathfrak F_{\tilde{\mathfrak{t}}_{m-1}^R}\Big]\\&=\E\Big[\nabla\big[\Phi (W_{\tilde{\mathfrak{t}}_{m-1}^R}+W_{s-\tilde{\mathfrak{t}}_{m-1}^R})]\Phi (W_{\tilde{\mathfrak{t}}_{m-1}^R}+W_{s-\tilde{\mathfrak{t}}_{m-1}^R})\Big|W_{\tilde{\mathfrak{t}}_{m-1}^R}\Big]\\
%&=\nabla\big[\Phi (W_{\tilde{\mathfrak{t}}_{m-1}^R})]\Phi (W_{\tilde{\mathfrak{t}}_{m-1}^R})+\E\Big[\nabla\big[\Phi (W_{s-\tilde{\mathfrak{t}}_{m-1}^R})]\Phi (W_{s-\tilde{\mathfrak{t}}_{m-1}^R})\Big|W_{\tilde{\mathfrak{t}}_{m-1}^R}\Big]\\
%&=\nabla\big[\Phi (W_{\tilde{\mathfrak{t}}_{m-1}^R})]\Phi (W_{\tilde{\mathfrak{t}}_{m-1}^R})+(s-\tilde{\mathfrak{t}}_{m-1}^R)\E\Big[\nabla\big[\Phi (W_{1})]\Phi (W_{1})\Big]
%\end{align*}
%such that
%\begin{align*}
%\E\Big[\big(\nabla\big[\Phi W_s]\Phi W_s-\nabla\big[\Phi W_{\tilde{\mathfrak{t}}_m^R}]\Phi W_{\tilde{\mathfrak{t}}_m^R}\big)\Big|\mathfrak F_{\tilde{\mathfrak{t}}_{m-1}^R}\Big]=(s-\tilde{\mathfrak{t}}_m^R)\E\Big[\nabla\big[\Phi W_1]\Phi W_1\Big].
%\end{align*}
we conclude that
\begin{align*}
\E[{\mathrm{E}}^{4,1}]\leq\delta\E\big[\|\bfe_{m-1}\|_{L^2_x}^2\big]+c(\delta)R_2(R)^2{\tau}^3.
\end{align*}
Combining everything and choosing $\delta$ conveniently small we obtain
\begin{align*}
\E\bigg[&\bigg(\int_{\mt}|\bfe_{m}|^2 \dx +\tfrac{1}{2}\int_{\mt}|\bfe_{m}-\bfe_{m-1}|^2 \dx+\tilde{\tau}_{m}^R\tfrac{\mu}{2}\int_{\mt}|\nabla\bfe_{m}|^2\dx\bigg)\bigg]\\
&\leq \E\bigg[\int_{\mt}|\bfe_{m-1}|^2 \dx\bigg]
+cR_2(R)^2\E\bigg[\tilde{\tau}_{m}^R \int_{\mt}|\bfe_{m-1}|^2 \dx+\tilde{\tau}_{m}^R \int_{\mt}|\bfe_{m}|^2 \dx\bigg]\\
&+c {\tau}^2R_2(R)^2 \E\bigg[ \int_{\tilde{\mathfrak{t}}_{m-1}^R}^{\tilde{\mathfrak{t}}_m^R} \Vert \nabla \partial_t {\bf y}(s)\Vert^2_{L^2_x}\, {\rm d}s \bigg]+c{\tau}^{3}e^{cR_2(R)^4}.
\end{align*}
Iterating this inequality and noticing that 
\begin{align*}
\sum_{n=1}^M\E\bigg[ \int_{\tilde{\mathfrak{t}}_{n-1}^R}^{\tilde{\mathfrak{t}}_n^R} \Vert \nabla \partial_t {\bf y}(s)\Vert^2_{L^2_x}\, {\rm d}s \bigg]\leq \E\bigg[ \int_{0}^{\mathfrak{t}_R} \Vert \nabla \partial_t {\bf y}(s)\Vert^2_{L^2_x}\, {\rm d}s \bigg]\leq\,cR^4,
\end{align*}
by Corollary \ref{cor:add}
as well as
$\bfe_0=0$ yields
\begin{align*}
\max_{1\leq n\leq M}\E&\bigg[\int_{\mt}|\bfe_{n}|^2 \dx\bigg] +\E\bigg[\sum_{n=1}^M\tilde{\tau}_{n}^R\int_{\mt}|\nabla\bfe_{n}|^2\dx\bigg]\\
&\leq cK(R)^2\E\bigg[\sum_{n=1}^M\tilde{\tau}_{n}^R\int_{\mt}|\bfe_{n}|^2 \dx\bigg]+c{\tau}^{2}e^{cR_2(R)^4},
%&+c {\tau}^2   \int_{0}^{T} \Vert \nabla \partial_t {\bf y}(s)\Vert^2_{L^2_x}\, {\rm d}s \bigg(\sup_{s \in [0,T]}\Vert \nabla{\bf y}(s)\Vert^2_{L^2_x}+\max_{1\leq n\leq M}\|\bfe_n\|_{L^2_x}^2\bigg).
\end{align*}
Applying Gronwall's lemma shows
\begin{align*}
\max_{1\leq n\leq M}\E&\bigg[\int_{\mt}|\bfe_{n}|^2 \dx\bigg] +\E\bigg[\sum_{n=1}^M\tilde{\tau}_{n}^R\int_{\mt}|\nabla\bfe_{n}|^2\dx\bigg]\\
&\leq c{\tau}^{2}e^{cK(R)^4},
%&+c {\tau}^2   \int_{0}^{T} \Vert \nabla \partial_t {\bf y}(s)\Vert^2_{L^2_x}\, {\rm d}s \bigg(\sup_{s \in [0,T]}\Vert \nabla{\bf y}(s)\Vert^2_{L^2_x}+\max_{1\leq n\leq M}\|\bfe_n\|_{L^2_x}^2\bigg).
\end{align*}
which gives the claim.
%Taking the maximum with respect to $m$, applying Gronwall's lemma and expectations proves
%\begin{align*}
%\E\bigg[&\mathbf{1}_{\Omega^\varepsilon_{\tilde{\tau}_{m}^R}}\bigg(\max_{1\leq m\leq M}\|\bfe_{m}\|_{L^2_x}^2+\sum_{m=1}^M \tilde{\tau}_{m}^R \|\nabla\bfe_{m}\|_{L^2_x}^2\bigg)\bigg]\\&\leq \,c\,{\tau}^{2-2\varepsilon}+c\,{\tau}^{2-2\varepsilon}\E\bigg[\int_{0}^{T} \Vert \nabla \partial_t {\bf y}(s)\Vert^2_{L^2_x}\, {\rm d}s \bigg(\sup_{s \in [0,T]}\Vert \nabla{\bf y}(s)\Vert^2_{L^2_x}+\max_{1\leq n\leq M}\|\bfe_n\|_{L^2_x}^2\bigg)\bigg].
%\end{align*}
%The last term is bounded by
%\begin{align*}
%\bigg(\E\bigg[\int_{0}^{T} \Vert \nabla \partial_t {\bf y}(s)\Vert^2_{L^2_x}\, {\rm d}s\bigg]^2\bigg)^{\frac{1}{2}} \bigg(\E\bigg[\sup_{s \in [0,T]}\Vert \nabla{\bf y}(s)\Vert^2_{L^2_x}+\max_{1\leq n\leq M}\|\bfe_n\|_{L^2_x}^2\bigg]^2\bigg)^\frac{1}{2},
%\end{align*}
%which can be controlled my means of Corollary \ref{cor:add} and Lemma \eqref{lem-rand1}. The claimed estimate follows.
\end{proof}

\section{Optimal weak order in probability for LBB-stable space-time discretisations of (\ref{eq:SNS})}\label{sec:fe}

Section \ref{sec:add}  validates the (strong) order (up to) $1$ in probability for semi-discretisation (\ref{tdiscr-add}) 
%of form
%\begin{equation}\label{tdiscr-add1}
%\frac{{\bf y}_{m} - {\bf y}_{m-1}}{\tau} - \mu {\mathcal A} {\bf y}_m = \mu {\mathcal A}[\Phi W(t_m)] - {\mathcal P}\bigl[ (\nabla {\bf y}_{m}){\bf y}_{m-1} \bigr]+{\mathcal P}\bigl[ \mathcal L^m}(\bfy_{m-1},\bfy_m)\bigr]\,,
%\end{equation}
for $W^{1,2}_{0,\Div}(\mt;\R^2)$-valued iterates $({\bf y}_m)_m$; the
%, where 
%$${\mathcal L}^m (\bfy_{m-1},\bfy_m) = \nabla [\Phi W(t_m)]{\bf y}_{m-1} + \nabla {\bf y}_m\bigl( \Phi W(t_{m-1})\bigr) - \nabla [\Phi W(t_m)] \Phi W(t_{m-1})\, .$$
scheme was set up to confine to approximating only the {\em timely regular part} ${\bf y}$ of the strong solution ${\bf u} = {\bf y} + \Phi W$ of (\ref{eq:SNS}). We recall that ${\bf u}_m := {\bf y}_m  + \Phi W(t_m)$ solves (\ref{tdiscr'}).
We now propose and analyse an implementable spatial discretisation of (\ref{tdiscr-add}) based on the finite element method. 

The study of temporal semi-discretisations in Section \ref{sec:add} involves exactly divergence-free iterates, and also the used test functions are from the same space $W^{1,2}_{0,\Div}(\mt,\R^2)$, which allowed to eliminate the pressure from the problem. This is not the case any more in the setting of a spatio-temporal discretisation --- {\em e.g.}~via well-known LBB-stable mixed finite element pairings $({V}^{h}(\mt;\R^2), P^{h}(\mt))$ on quasi-uniform meshes ${\mathscr T}_h$ of size $0 < h \ll 1$  covering ${\mathcal O}$, which, in particular, involve the space $P^h(\mt)$ for discrete pressures. 

We work with a standard finite element set-up for incompressible fluid mechanics, see {\em e.g.}~\cite{GR}.
We denote by $\mathscr{T}_h$ a quasi-uniform subdivision of $\mt$ into triangles of maximal diameter $h>0$.   
For $K\subset \mt$ and $\ell\in \setN _0$ we denote by
$\mathscr{P}_\ell(K)$ the polynomials on $K$ of degree less than or equal
to $\ell$. 
%Moreover, we set $\mathscr{P}_{-1}(\mathcal S):=\set {0}$. 
Let us
characterize the finite element spaces $V^h(\mt;\R^2)$ and $P^h(\mt)$ as
\begin{align*}%\label{def:Vh}
  V^h(\mt;\R^2)&:= \set{\bfV_h \in W^{1,2}_0(\mt;\R^2)\,:\, \bfV_h|_{K}
    \in \mathscr{P}_i(K;\R^2)\,\,\forall K\in \mathscr T_h},\\
P^h(\mt)&:=\set{R_h \in L^{2}(\mt)/ \R\,:\, R_h|_{K}
    \in \mathscr{P}_j(K)\,\,\forall K\in \mathscr T_h},
\end{align*}
for some $i,j\in\N_0$.
%
% to get \eqref{eq:stab'} below.
In order to guarantee stability of our approximations we relate $V^h(\mt;\R^2)$ and $P^h(\mt)$ by the discrete LBB-condition, that is we assume that
\begin{align*}
\sup_{\bfV_h\in V^h(\mt;\R^2)} \frac{\int_{\mt}\Div\bfV_h\,R_h\dx}{\|\nabla\bfV_h\|_{L^2_x}}\geq\,C\,\|R_h\|_{L^2_x}\quad\,\forall R_h\in P^h(\mt),
\end{align*}
where $C>0$ does not depend on $h$. This gives a relation between $i$ and $j$ (for instance the choice $(i,j)=(1,0)$ is excluded, whereas $(i,j)=(2,0)$ is allowed).
Finally, we define the space of discretely solenoidal finite element functions by 
\begin{align*}%\label{def:Vh}
  V^h_{\Div}(\mt;\R^2)&:= \bigg\{\bfV_h\in V^h(\mt;\R^2):\,\,\int_{\mt}\Div\bfV_h\,R_h\dx=0\,\quad\forall R_h\in P^h(\mt)\bigg\}.
\end{align*}
Let $\Pi_h:L^2(\mt;\R^2)\rightarrow V_{\Div}^h(\mt;\R^2)$ be the $L^2(\mt;\R^2)$-orthogonal projection onto $V_{\Div}^h(\mt;\R^2)$; see \cite{HeyRa1}, for instance. 
%The following results concerning the approximability of $\Pi_h$ are well-known (see, for instance \cite{HeyRa1}). There is $c>0$ independent of $h$ such that we have
%  \begin{align}
%    \label{eq:stab}
% \int_{\mt} \Big|\frac{\bfv-\Pi_h \bfv}{h}\Big|^2\dx+ \int_{\mt} \abs{\nabla\bfv-\nabla\Pi_h \bfv}^2\dx &\leq
%    \,c\, \int_{\mt} \abs{\nabla
%      \bfv}^2\dx
%  \end{align}
%for all $\bfv\in W^{1,2}_{0,\Div}(\mt)$, and
%  \begin{align}
%    \label{eq:stab'}
% \int_{\mt} \Big|\frac{\bfv-\Pi_h \bfv}{h}\Big|^2\dx+ \int_{\mt} \abs{\nabla\bfv-\nabla\Pi_h \bfv}^2\dx &\leq
%    \,c\,h^2 \int_{\mt} \abs{\nabla^2
%      \bfv}^2\dx
%  \end{align}
%for all $\bfv\in W^{2,2}\cap W^{1,2}_{0,\Div}(\mt)$. Similarly, if $\Pi_h^\pi:L^2(\mt)/\R\rightarrow P^h(\mt)$ denotes the $L^2(\mt)$-orthogonal projection onto $P^h(\mt)$, and we have
%\begin{align}
%\label{eq:stabpi}
% \int_{\mt} \Big|\frac{p-\Pi_h^\pi p}{h}\Big|^2\dx &\leq
%    \,c\, \int_{\mt} \abs{\nabla
%      p}^2\dx
%\end{align}
%for all $p\in W^{1,2}(\mt)/\R$, and
%\begin{align}
%\label{eq:stabpi'}
% \int_{\mt} \Big|\frac{p-\Pi_h^\pi p}{h}\Big|^2\dx &\leq
%    \,c\,h^2 \int_{\mt} \abs{\nabla^2
%      p}^2\dx
%\end{align}
We are now ready to define an LBB-stable mixed FEM for (\ref{tdiscr-add}) based on mixed finite elements: Let ${\bf Y}_0 \in V^h_{\Div}(\mt;\R^2)$. For any $m \geq 1$, find
 $V_{\rm div}^h({\mathcal O};\R^2) \times P^h({\mathcal O})$-valued random variables $({\bf Y}_m, P_m)$ such that ${\mathbb P}$-a.s.
\begin{align}\label{tdiscr-add-fem1}
&\int_{\mt} \bigl({\bf Y}_{m}- {\bf Y}_{m-1}\bigr) \cdot \pmb{\varphi} \dx  + \mu \tau \int_{\mt}\nabla {\bf Y}_m: \nabla \pmb{\varphi} \dx 
 - \tau \int_{\mt}  P_m \cdot {\rm div}\, \pmb{\varphi} \dx\\
& \quad  =\nonumber
  \tau \int_{\mt} \Big(\mu \nabla [\Phi W(t_m)] : \nabla \pmb{\varphi} 
  - \bigl((\nabla {\bf Y}_{m}){\bf Y}_{m-1} + \frac{1}{2}[{\rm div}\, {\bf Y}_{m-1}] {\bf Y}_m\bigr) \cdot \pmb{\varphi}\Big)\dx\\
 &\qquad + \int_{\mt}
  \widetilde{\mathcal L^m}(\bfY_{m-1},\bfY_m)  \cdot \pmb{\varphi} \dx\nonumber
 , \\ \label{tdiscr-add-fem2}
&\int_{\mt} {\rm div} {\bf Y}_{m} \, R \dx = 0,
\end{align}
for all $(\pmb{\varphi}, R) \in V_{\rm div}^h({\mathcal O};\R^2) \times P^h({\mathcal O})$, where the following map approximates $\mathcal L^{m}$ from (\ref{tdiscr-add}),
\begin{align}\nonumber\widetilde{\mathcal L^m}(\bfY_{m-1},\bfY_m)  &=\nabla \Pi_h[\Phi W(t_m)]{\bf Y}_{m-1} + \nabla {\bf Y}_m \Pi_h\bigl( \Phi W(t_{m-1})\bigr) \\ \label{midi1}
& - \nabla \Pi_h\bigl[\Phi W(t_m)\bigr] \Pi_h \Phi W(t_{m-1}) + \frac{1}{2} [{\rm div}\, {\bf Y}_{m-1}] \Pi_h[\Phi W(t_m)]\\ 
&\nonumber
+ \frac{1}{2} \Bigl[{\rm div}\, \Pi_h \bigl[\Phi W(t_{m-1})\bigr]\Bigr] \Pi_h \bigl[\Phi W(t_m)\bigr]
+ \frac{1}{2}\Bigl [ {\rm div}\, \Pi_h\bigl[\Phi W(t_{m-1})\bigr]\Bigr]{\bf Y}_m
\, .
\end{align}
The additional term
$\frac{1}{2}\int_{\mt} [{\rm div}\, {\bf Y}_{m-1}] {\bf Y}_m  \cdot \pmb{\varphi}\dx$ in (\ref{tdiscr-add-fem1})
is used to control nonlinear effects in the presence of {\em discretely} divergence-free velocity iterates, as do the last three terms in (\ref{midi1}). The reason why the scheme is set up in this form will become clear if we consider the $V_{\Div}^h(\mt;\R^2)$-valued random variable
\begin{equation}\label{fem-9}
{\bf U}_m := {\bf Y}_m + \Pi_h \bigl[\Phi W(t_{m})\bigr]\,,
\end{equation}  which solves the following equation
\begin{align}\label{txdiscrpia}
\begin{aligned}
\int_{\mt}\bfU_{m}\cdot\bfvarphi \dx &+\tau\int_{\mt}\big((\nabla\bfU_{m})\bfU_{m-1}+(\Div\bfU_{m-1})\bfU_{m}\big)\cdot\bfphi\dx\\
&+\mu\,\tau\int_{\mt}\nabla\bfU_{m}:\nabla\bfphi\dx-\tau\int_{\mt}P_{m}\,\Div\bfvarphi\dx\\&=\int_{\mt}\bfU_{m-1}\cdot\bfvarphi \dx+\int_{\mt}\Phi\,\Delta_mW\cdot \bfvarphi\dx
%, \\
%\int_{\mt} {\rm div} {\bf U}_{m} \cdot \chi&\dx = 0\, ,
\end{aligned}
\end{align}
for all $\pmb{\varphi}  \in V_{\rm div}^h({\mathcal O};\R^2)$. This combined spatio-temporal discretisation with $V^h_{\Div}(\mt;\R^2)$-valued $({\bf U}_m)_m$ has been investigated in 
\cite{BrPr} in the more general context of multiplicative noise in (\ref{eq:SNS}); the key observation now  is that for the spatial error analysis for (\ref{txdiscrpia}) it suffices to bound $\max_m \Vert {\bf Y}_m - {\bf y}_m\Vert_{L^2_x}$ thanks to (\ref{fem-9}), which only involves spatially regular noise. Therefore, 
 we isolate from the convergence proof for \eqref{tdiscr-add-fem1}--\eqref{tdiscr-add-fem2} {\em those} arguments from 
\cite{BrPr} for (\ref{txdiscrpia}) which address the {\em spatial discretisation} only. We have teh following result.
%
%
% which address {\em spatial} discretisation effects only --- which are {\em not} crucially affected by the (spatially regular) noise ---, and to then transfer them to $({\bf Y}_m)_m$ by the above relation of these sequences. 
 \begin{theorem}\label{thm:maintdiscra}
Assume that $\bfu_0\in L^r(\Omega,W^{2,2}_{\Div}(\mt;\R^2))\cap L^{5r}(\Omega,W^{1,2}_{\Div}(\mt;\R^2))$ for some $r\geq8$ and that $\Phi\in L_2(\mathfrak U;W^{1,2}_{0,\Div}\cap W^{3,2}(\mt;\R^2))$. Let $$(\bfu,(\mathfrak{t}_R)_{R\in\N},\mathfrak{t})$$ be the unique global maximal strong solution to \eqref{eq:SNS} from Theorem \ref{thm:inc2d}.
Then we have for any $\xi>0$, $\alpha<1$ 
\begin{align*}
&\max_{1\leq m\leq M}\mathbb P\bigg[\|\bfy(t_m)-\bfY_{m}\|_{L^2_x}^2+\sum_{n=1}^m \tau\|\nabla\bfy(t_n)-\nabla\bfY_{n}\|_{L^2_x}^2>\xi\,(\tau^{2\alpha} + h^{2\alpha})\bigg]\rightarrow 0
\end{align*}
as $\tau,h\rightarrow0$,
where $\bfy$ is the solution
to \eqref{eq:SNSy} and $(\bfY_{m})_{m=1}^M$ is the solution to \eqref{tdiscr-add-fem1}--\eqref{tdiscr-add-fem2}.
\end{theorem}
Let $\tilde\bfu_0^R:=\bfu_0$.
To prepare for the proof of Theorem \ref{thm:maintdiscra} we introduce
%, we recall (\ref{tdiscrR}) with {\color{green} $\tilde{\tau}_{m}^R:=\mathfrak t_m^R-\mathfrak t_{m-1}^R$}, and 
the $W^{1,2}_{0,\Div}(\mt;\R^2)$-iterates $(\tilde{\bfu}_{m}^R)_{m=1}^M$ solving for $m\geq1$
\begin{align}\label{tdiscrR}
\begin{aligned}
\int_{\mt}&\tilde{\bfu}_{m}^R\cdot\bfvarphi \dx +\tilde{\tau}_{m}^R\int_{\mt}(\nabla\tilde{\bfu}_{m}^R)\bfu^R_{m-1}:\nabla\bfphi\dx\\
&+\mu\,\tilde{\tau}_{m}^R\int_{\mt}\nabla\tilde{\bfu}_{m}^R:\nabla\bfphi\dx=\int_{\mt}\tilde{\bfu}_{m-1}^R\cdot\bfvarphi \dx+\frac{\tilde{\tau}_{m}^R}{\tau}\int_{\mt}\Phi\,\Delta_m W\cdot \bfvarphi\dx,
\end{aligned}
\end{align}
for every $\bfphi\in W^{1,2}_{0,\Div}(\mt;\R^2)$,
%Obviously $\tilde{\bfu}_{m}^R=\bfu_{m}$ in $[t_m\leq  \mathfrak t_{m}^R]$.
where
$\tilde{\bfu}_{m}^R=\bfu_{m}$ in $\{t_m\leq \tilde{\mathfrak t}_{m}^R\}$ with $\tilde{\mathfrak t}_{m}^R$ and $\tilde{\tau}_{m}^R$ introduced in \eqref{eq:tmR}. Accordingly, we define
$V^h_{\Div}({\mathcal O};\R^2)$-valued iterates $(\tilde{\bfU}_{m}^R)_m$ solving
\begin{align}\label{tdiscrRd}
\begin{aligned}
\int_{\mt}&\tilde{\bfU}_{m}^R\cdot\bfvarphi \dx +\tilde{\tau}_{m}^R\int_{\mt}\big((\nabla\tilde\bfU^R_{m})\tilde\bfU^R_{m-1}+(\Div\tilde\bfU^R_{m-1})\tilde\bfU^R_{m}\big)\cdot\bfphi\dx\\
&+\mu\,\tilde{\tau}_{m}^R\int_{\mt}\nabla\tilde{\bfU}_{m}^R:\nabla\bfphi\dx=\int_{\mt}\tilde{\bfU}_{m-1}^R\cdot\bfvarphi \dx+\frac{\tilde{\tau}_{m}^R}{\tau}\int_{\mt}\Phi\,\Delta_m W\cdot \bfvarphi\dx,
\end{aligned}
\end{align}
for every $\bfphi\in  V_{\rm div}^h({\mathcal O};\R^2)$. Again, $\tilde{\bfU}_{m}^R=\bfU_{m}$ in $[t_m\leq \tilde{\mathfrak t}_{m}^R]$.

Next, by setting $R=c^{-1/4}\sqrt[4]{-\varepsilon \log  h}$  where $\varepsilon>0$ is arbitrary, we have for any $\xi>0$
\begin{align*}
&\mathbb P\bigg[\max_{1\leq m\leq M}\|\bfu_m -\bfU_{m}\|_{L^2_x}^2+\sum_{m=1}^M \tau\|\nabla\bfu_m -\nabla\bfU_{m}\|_{L^2_x}^2>\xi\, h^{2-2\varepsilon}\bigg]\\
&\quad \leq \mathbb P\bigg[\max_{1\leq m\leq M}\| \tilde{\bfu}_{m}^R-\tilde{\bfU}_{m}^R\|_{L^2_x}^2+\sum_{m=1}^M \tilde{\tau}_{m}^R\|\nabla\tilde{\bfu}_{m}^R - \nabla\tilde{\bfU}_{m}^R\|_{L^2_x}^2>\xi\, h^{2-2\varepsilon} \bigg]\\
&\qquad +\mathbb P\big[\{\tilde{\mathfrak t}^{\tt d}_R<T\}\big]\, .
\end{align*}
The last term vanishes $\p$-a.s.~by \eqref{eq:2008} as $R \rightarrow \infty$; for the first term to converge to zero for $h \rightarrow 0$, we use Markov's inequality, and the bound
\begin{equation}\label{etim1} {\mathbb E}\bigg[\max_{1\leq m\leq M}\|\tilde\bfu^R_m -\tilde\bfU^R_{m}\|_{L^2_x}^2+\sum_{m=1}^M \tilde\tau_m^R\|\nabla\tilde\bfu^R_m -\nabla\tilde\bfU^R_{m}\|_{L^2_x}^2\bigg] 
\leq  c e^{cR^4}h^2,
\end{equation}
whose proof will be sketched.
Since the error between $\Pi_h\Phi W$ (from \eqref{fem-9}) and $\Phi W$ (from \eqref{identd1}) can be controlled by the assumed regularity of $\Phi$ and the approximation properties of $\Pi_h$, we conclude that
\begin{align*}
&\mathbb P\bigg[\max_{1\leq m\leq M}\|\bfy_m -\bfY_{m}\|_{L^2_x}^2+\sum_{m=1}^M \tau\|\nabla\bfy_m -\nabla\bfY_{m}\|_{L^2_x}^2>\xi\, h^{2-2\varepsilon}\bigg]\rightarrow0,
\end{align*}
which yields
the claim of Theorem \ref{thm:maintdiscra} due to Theorem \ref{thm:3.1tilde}.

We are left with the proof of \eqref{etim1}.
%as $h,\tau\rightarrow0$ (recall that ${\color{red} \mathfrak t^{\tt d}_R}\rightarrow \infty$ $\p$-a.s.~by Lemma \ref{lemma:3.1a} and that $\mathfrak t_M^R<t_M$ implies $\mathfrak t_R<T$).
%Relabeling $\alpha$ we have proved the following result.
We define the error $\bfE^R_{m}= \tilde{\bfu}_{m}^R-\tilde{\bfU}_{m}^R$. Subtracting \eqref{tdiscrRd} from 
\eqref{tdiscrR} and using the additive character of the noise gives
%\begin{align*}%\label{tdiscr'}
%\begin{aligned}
%\int_{\mt}&\bfe_{h,m}\cdot\bfvarphi \dx +\int_{\mathfrak t_{m-1}^R}^{\mathfrak t_m^R}\int_{\mt}\mu\nabla\bfu(\sigma):\nabla\bfphi\dx\ds-\tilde{\tau}_{m}^R\int_{\mt}\mu\nabla\bfu^R_{h,m}:\nabla\bfphi\dx\\&=\int_{\mt}\bfe_{h,m-1}\cdot\bfvarphi \dx-\int_{\mathfrak t_{m-1}^R}^{\mathfrak t_m^R}\int_{\mt}(\nabla\bfu(\sigma))\bfu(\sigma)\,\bfphi\dxs\\&+\tilde{\tau}_{m}^R\int_{\mt}\big((\nabla\bfu_{h,m}^R)\bfu_{h,m-1}^R+(\Div\bfu^R_{h,m-1})\bfu_{h,m}^R\big)\cdot\bfphi\dx\\
%&+\int_{\mt}\int_{\mathfrak t_{m-1}^R}^{\mathfrak t_m^R}\Phi(\bfu(\sigma))\,\dd W\cdot \bfvarphi\dx-\int_{\mt}\int_{\mathfrak t_{m-1}^R}^{\mathfrak t_m^R}\Phi(\bfu_{h,m-1}^R)\,\dd W\cdot \bfvarphi\dx\\
%&-\int_{\mt}\int_{\mathfrak t_{m-1}^R}^{\mathfrak t_m^R}\nabla\Delta^{-1}\Div\Phi(\bfu(\sigma))\,\dd W\cdot \bfvarphi\dx+\int_{\mathfrak t_{m-1}^R}^{\mathfrak t_m^R}\int_{\mt}\pi_{\mathrm{det}}(\sigma)\,\Div\bfphi\dx\ds
%\end{aligned}
%\end{align*}
for every $\bfphi\in V_{\Div}^h(\mt;\R^2)$
\begin{align}\label{eqn-BB}
\begin{aligned}
\int_{\mt}&\tilde{\bfE}_{m}^R\cdot\bfvarphi \dx + \mu \tilde{\tau}_{m}^R\int_{\mt} \nabla\tilde{\bfE}_{m}^R:\nabla\bfphi\dx\\&=\int_{\mt}\bfE^R_{m-1}\cdot\bfvarphi \dx + {\tilde{\tau}_{m}^R \int_{\mt}\tilde p^R_m \,\Div\bfphi\dx}
%+\int_{\mathfrak t_{m-1}^R}^{\mathfrak t_m^R}\int_{\mt}\mu\big(\nabla\bfu(\mathfrak t_{m}^R)-\nabla\bfu(\sigma)\big):\nabla\bfphi\dx\ds
\\
%&+\int_{\mathfrak t_{m-1}^R}^{\mathfrak t_m^R}\int_{\mt}\Big((\nabla\bfu(\mathfrak t_{m}^R))\bfu(\mathfrak t_{m}^R)-(\nabla\bfu(\sigma))\bfu(\sigma)\Big)\cdot\bfphi\dxs\\
&\quad -\tilde{\tau}_{m}^R\int_{\mt}\Big((\nabla\tilde{\bfu}_{m}^R) \tilde{\bfu}_{m-1}^R)-\big((\nabla\tilde\bfU^R_{m})\tilde\bfU^R_{m-1}+(\Div\tilde\bfU^R_{m-1})\tilde\bfU^R_{m}\big)\Big)\cdot\bfphi\dx\, ,
%&+\int_{\mt}\int_{\mathfrak t_{m-1}^R}^{\mathfrak t_m^R}\big(\Phi(\bfu(\sigma))-\Phi(\bfu^R_{h,m-1})\big)\,\dd W\cdot \bfvarphi\dx\\
%-\int_{\mt}\int_{\mathfrak t_{m-1}^R}^{\mathfrak t_m^R}\nabla\Delta^{-1}\Div\Phi(\bfu(\sigma))\,\dd W\cdot \bfvarphi\dx
\end{aligned}
\end{align}
with
\begin{align*}%\label{tdiscrtilde}
\tilde{p}_m^R:=\begin{cases} p_m,\quad \text{in}\quad [t_m=\tilde{\mathfrak t}_m^R],\\
p_{\mathfrak m_{R}},\quad \text{in}\quad [t_m> \tilde{\mathfrak{t}}_m^R],
\end{cases}
\end{align*}
where $\mathfrak m_R$ is defined below \eqref{eq:tRd}.
Here $p_m$ from (\ref{tdiscrRp1}) enters due to only {\em discretely} solenoidal test-functions.
Setting $\bfphi=\Pi_h\tilde{\bfE}_{m}^R$, standard manipulations then lead to
\begin{align*}%\label{eq:0207}
\int_{\mt}&\frac{1}{2}\big(|\Pi_h\tilde{\bfE}_{m}^R|^2-|\Pi_h\tilde{\bfE}_{m-1}^R|^2+|\Pi_h\tilde{\bfE}_{m}^R-\Pi_h\tilde{\bfE}_{m-1}^R|^2\big) \dx+\mu\tilde{\tau}_{m}^R\int_{\mt}|\nabla\tilde\bfE^R_{m}|^2\dx\\
&=-\mu \tilde{\tau}_{m}^R\int_{\mt}\nabla\tilde\bfE^R_{m}:\nabla\big(\tilde{\bfu}_{m}^R-\Pi_h\tilde{\bfu}_{m}^R)\big)\dx  + {\tilde{\tau}_{m}^R \int_{\mt}\tilde p^R_m\,\Div \Pi_h\tilde{\bfE}_{m}^R\dx}\\
%&+\mu \int_{\mathfrak t_{m-1}^R}^{\mathfrak t_{m}^R}\int_{\mt}\big(\nabla\bfu(\mathfrak t_{m}^R)-\nabla\bfu(\sigma)\big):\nabla\Pi_h\bfe_{h,m}\dx\ds
%\\
%&+\int_{\mathfrak t_{m-1}^R}^{\mathfrak t_{m}^R}\int_{\mt}\Big((\nabla\bfu(\mathfrak t_{m}^R))\bfu(\mathfrak t_{m-1}^R)-(\nabla\bfu(\sigma))\bfu(\sigma)\Big)\cdot\Pi_h\bfe_{h,m}\dxs\\
&\quad -\tilde{\tau}_{m}^R\int_{\mt}\Big((\nabla\tilde{\bfu}_{m}^R) \tilde{\bfu}_{m-1}^R)-\big((\nabla\tilde\bfU^R_{m})\tilde\bfU^R_{m-1}+(\Div\tilde\bfU^R_{m-1})\tilde\bfU^R_{m}\big)\Big)\cdot \Pi_h\tilde{\bfE}_{m}^R\dx\\
%&+\int_{\mt}\int_{\mathfrak t_{m-1}^R}^{\mathfrak t^R_m}\big(\Phi(\bfu(\sigma))-\Phi(\bfu^R_{h,m-1})\big)\,\dd W\cdot \Pi_h\bfe_{h,m}\dx\\
%&-\int_{\mt}\int_{\mathfrak t_{m-1}^R}^{\mathfrak t_{m}^R}\nabla\Delta^{-1}\Div\Phi(\bfu(\sigma))\,\dd W\cdot \Pi_h\bfe_{h,m}\dx\\
%&\quad +\int_{\mathfrak t_{m-1}^R}^{\mathfrak t_{m}^R}\int_{\mt}\pi_{\mathrm{det}}(\sigma)\,\Div\Pi_h\bfe_{h,m}\dxs\\
&=:I_1(m)+\dots +I_3(m)\, .
\end{align*}
We use well-known approximation properties of $\Pi_h$ to bound
\begin{equation*}
I_1(m)\leq  \frac{\mu}{2}\tilde{\tau}_{m}^R\int_{\mt}|\nabla\tilde\bfE^R_{m}|^2\dx+c\tilde{\tau}_{m}^Rh^2\int_{\mt}|\nabla^2\tilde{\bfu}_{m}^R|^2\dx\,,
\end{equation*}
and (\ref{lem:3.1b}) may now be applied to optimally bound the last term. A corresponding longer, but elementary estimation which rests on the same ingredients bounds $I_3(m)$; see also \cite[bounds for term $I_4(m)$ in the proof of Thm.~4.2]{BrPr}. For $I_2(m)$, we use well-known approximation properties for the $L^2$-projection $Q_h: L^2({\mathcal O}) \rightarrow P^h(\mt)$, and (\ref{tdiscr-add-fem2}),  and (\ref{tdiscrRp}) to conclude
\begin{align}\nonumber
I_2(m) &= \tilde{\tau}_{m}^R \int_{\mt} \bigl(\tilde p^m - Q_h \tilde p^m)\,\Div \Pi_h\tilde{\bfE}_{m}^R\dx\\
&\leq \kappa\tilde{\tau}_{m}^R\int_{\mt}|\nabla\tilde\bfE^R_{m}|^2\dx +c(\kappa)\tilde{\tau}_{m}^Rh^2\int_{\mt}|\nabla \tilde p_{m}^R|^2\dx\,;
\end{align}
see again \cite[bound for term $I_7(m)$ in the proof of Thm.~4.2]{BrPr}. The expectation of the sum of the leading term is bounded by $c e^{cR^4}$, due to (\ref{tdiscrRp}).
The proof of \eqref{etim1} (and hence that of Theorem \ref{thm:maintdiscra}) is thus complete.

\begin{remark}\label{rem:A}
For the stochastic Stokes equation instead of \eqref{eq:SNS}, and a corresponding simplification of
               scheme \eqref{tdiscr-add-fem1}--\eqref{tdiscr-add-fem2}, we have for all $\alpha<1$
               \begin{align}
 \label{new1}              \max_{1 \leq m \leq M} {\mathbb E}\biggl[ \Vert \bfy(t_m) - {\bf Y}_m\Vert^2_{L^2} +\sum_{n=1}^m \tau\|\nabla\bfy(t_n) -\nabla\bfY_{n}\|_{L^2_x}^2\biggr]
                \leq c (\tau^{2 } + h^{2}) .
                \end{align}
The first part of \eqref{new1}, which addresses semi-discretisation in time results, starts from error identity \eqref{erro1}
in simplified form -- without the term $\mathrm{NLT}_m$. Furthermore, no (discrete) stopping times are needed, the stopping times in the counterparts of Lemmas \ref{lem:regadditive}, \ref{lem-rand1}, and 
Corollary \ref{cor:add} are all equal to $T$. The second part of \eqref{new1} on semi-discretisation in space starts from
error identity \eqref{eqn-BB}
in simplified form, without the last term that addresses the role of the nonlinearity. Again, involved 
stopping times are all $T$, which settles \eqref{new1}. 
\end{remark}


\begin{thebibliography}{[M]}
\bibitem{BeMi1} H.~Bessaih, A.~Millet: Strong $L^2$ convergence of time numerical schemes for the stochastic 2D Navier-Stokes equation, IMA Journal of Numerical Analysis, 39-4 (2019), pp.~2135--2167.
\bibitem{BeMi1a} H.~Bessaih, A.~Millet:  Space-time Euler discretization schemes for the stochastic 2D Navier-Stokes equations, arXiv: 2004.04932v1 (2020)
\bibitem{BeMi2} H.~Bessaih, A.~Millet: Strong rates of convergence of space-time discretization schemes for the 2D Navier-Stokes equations with additive noise, arXiv:2102.01162.
%\bibitem{BrDe} C.-E.~Br\'ehier, A.~Debussche: \emph{Kolmogorov equations and weak order analysis for SPDEs with nonlinear diffusion coefficient.} J.~Math.~Pures Appl.~119, pp.~193--254. (2018)
%\bibitem{BrGo} C.-E.~Br\'ehier, L. gouden\'{e}ge: \emph{Weak convergence rates of splitting schemes for the stochastic Allen-Cahn equation.} BIT Num.~Math.~60, pp.~543--582. (2020) 
\bibitem{BrDo} D.~Breit \& A.~Dodgson: \emph{Convergence rates for the numerical approximation of the 2D stochastic Navier--Stokes equations.} Numer.~Math.~147, 553--578. (2021)
\bibitem{BrPr} D.~Breit, A.~Prohl: \emph{Error analysis for 2D stochastic Navier--Stokes equations in bounded domains.} arXiv:2109.06495v1
\bibitem{BCP}
Z.~Brze\'zniak, E.~Carelli, A.~Prohl (2013): \emph{Finite-element-based discretizations of the incompressible Navier--Stokes
equations with multiplicative random forcing}, IMA J.~Num.~Anal.~33, pp.~771--824.
%\bibitem{ChCh} Chen, Jianwen; Chen, Zhi-Min (2010): Stochastic non-Newtonian fluid motion equations of a nonlinear bipolar viscous fluid. J. Math. Anal. Appl. 369, no. 2, 486--509.
%\bibitem[DeSa]{DeSa}Gabriel Deugoue, Mamadou Sango (2011): Weak solutions to stochastic 3D Navier--Stokes--$\alpha$ model of turbulence:
%$\alpha$--Asymptotic behavior. J. Math. Anal. Appl. 384, 49--62
\bibitem{CP}
E.~Carelli, J. A.~Prohl (2012): \emph{Rates of convergence for discretizations of the stochastic incom- pressible Navier-Stokes equations.} SIAM J.~Numer.~Anal.~50(5), pp.~2467--2496.
\bibitem{Ca} M.~Capi\'nski, A note on uniqueness of stochastic Navier-Stokes equations, Univ.~Iagell. Acta Math.~30 (1993), pp.~219--228.

\bibitem{CC} M.~Capi\'nski, N.\,J.~Cutland, Stochastic Navier-Stokes equations, Acta Appl. Math. 25 (1991), pp.~59--85.
   \bibitem{PrZa} G.~Da Prato, J.~Zabczyk (1992): Stochastic Equations in Infinite Dimensions, Encyclopedia
Math.~Appl., vol.~44, Cambridge University Press, Cambridge. 
%\bibitem{De} A.~Debussche (2011): Weak approximation of stochastic partial differential equations: the nonlinear case. Math.~Comp. 273, pp.~89--117.
%\bibitem{DePr} A.~Debussche, J.~Printems (2009): weak order for the discretisation of the stochastic heat equation. Math.~Comp.~266, pp.~845--863.
%
%\bibitem{Do} P.~D\''orsek (2012): Semigroup splitting and cubature approxmiations for the stochastic Navier--Stokes equations. SIAM J.~Num.~Anal.~Vol.~50, No.~2, pp.~729--746.

\bibitem{FePrVo} X.~Feng, A.~Prohl, L.~Vo: Optimally convergent mixed finite element methods for the stochastic Stokes equations. IMA J.~Num.~Anal.~(2021).


%\bibitem{FlMa} F.~Flandoli, B.~Maslowski: \emph{Ergodicity of the 2D Navier--Stokes equation under random perturbations} Commun.~Math.~Phys.~171, pp.~119--141. (1995)


\bibitem{GR} V.~Girault, P.A.~Raviart: Finite element methods for Navier-Stokes equations, Springer, New York (1986).

\bibitem{GlZi} N.~Glatt-Holz, M.~Ziane: \emph{Strong Pathwise Solutions of the Stochastic Navier-Stokes System.} Adv.~Diff.~Equ.~14, 567--600. (2009)

\bibitem{HeyRa1} J.G.~Heywood, R.~Rannacher: \emph{Finite element approximation of the nonstationary Navier-Stokes problem. I.~Regularity of solutions and second-order error estimates for spatial discretization.} SIAM J.~Numer.~Anal.~19, 275--311. (1982).

%\bibitem{KP} P.E.~Kloeden, E.~Platen, Numerical Solution of Stochastic Differential Equations, Springer-Verlag, New York, 1992.

%\bibitem{KuVi} I. Kukavica, V. Vicol. On moments for strong solutions of the 2D stochastic Navier--Stokes equations in a bounded domain. Asympt. Anal. 90 (2014), no. 3-4, 189--206.

\bibitem{KukShi}
S.~Kuksin and A.~Shirikyan.
\newblock {\em Mathematics of two-dimensional turbulence}, volume 194 of {\em
  Cambridge Tracts in Mathematics}.
\newblock Cambridge University Press, Cambridge, 2012.


%
%\bibitem{LPS} G. J. Lord, C. Powell, T. Shardlow (2014): An Introduction to Computational Stochastic PDEs. Cambridge University Press. 520 p. (Cambridge Texts in Applied Mathematics)
%
%
%
%\bibitem{PrRo}  C. Pr\'{e}v\^{o}t, M. R\"ockner (2007): A concise course on stochastic partial differential equations. Lecture Notes in Mathematics, 1905. Springer, Berlin.

\bibitem{Mi} R.~Mikulevicius: \emph{On strong $H^1_
2$-solutions of stochastic Navier--Stokes equations in a bounded domain.} SIAM J. Math. Anal. 
Vol.~41, No.~3, pp.~1206--1230. (2009)

%\bibitem{Nu} D.~Nualart, The Malliavin calculus and related topics, Springer-Verlag, New York, 1995.

%\bibitem{Pr} J.~Printems: \emph{On the discretization in time of parabolic stochastic partial differential equations.} Math.~Mod.~Numer.~Anal.~35, pp.~1055--1078. (2001)

%\bibitem{simon} Simon, J. (1990) Sobolev, Besov and Nikolskii fractional spaces: imbeddings and comparisons for vector valued
%spaces on an interval. Ann. Math. Pura Appl., 157, 117--148.


%\bibitem{Ya1} Y. Yan (2004): Semidiscrete Galerkin Approximation for a Linear Stochastic Parabolic Partial Differential Equation Driven by an Additive Noise. Num. Math. 44, 829--847.
%
%\bibitem{Ya2} Y. Yan (2005): Galerkin finite element methods for stochastic parabolic partial differential equations,
%SIAM J. Numer. Anal., 43, pp. 1363--1384.



%\bibitem{Zh} X. Zhang (2010): Smooth solutions of non-linear stochastic partial differential equations driven by multiplicative noises. Sci. China Math. 53 (2010), 2949--2972.

\end{thebibliography}
\end{document}